\pdfoutput = 1
\documentclass{paper}
\usepackage{graphicx}
\usepackage{amsthm,amssymb,amscd,amsmath,latexsym,indentfirst,color,amsfonts}
\usepackage[mathcal]{eucal}
\usepackage{geometry,amsthm,graphics,tabularx,shapepar,float}
\usepackage{tikz-cd}
\usetikzlibrary{cd}
\usepackage{stackrel}
\usepackage{graphicx}
\usepackage{extarrows}
\usepackage[all]{xypic}
\usepackage{tikz}
\usepackage{bm}
\usepackage{extarrows}

\usetikzlibrary{cd}
\usetikzlibrary{calc}
\theoremstyle{plain}
\newtheorem{theorem}{Theorem}[section]
\newtheorem{prop}[theorem]{Proposition}
\newtheorem{coro}[theorem]{Corollary}
\newtheorem{lemma}[theorem]{Lemma}
\newtheorem{ex}[theorem]{Example}

\newtheorem{rem}[theorem]{Remark}

\newtheorem{defi}[theorem]{Definition}
\newtheoremstyle{defandprop}
{3pt}
{3pt}
{\itshape}
{}
{ }
{\thmname{#1}\thmnumber{ #2}.\thmnote{ (#3)}}
\theoremstyle{defandprop}
\newtheorem{defprop}[theorem]{Proposition and Definition}

\newcommand{\ble}{\begin {lemma}}
\newcommand{\ele}{\end {lemma}}
\newcommand{\bthm}{\begin {theorem}}
\newcommand{\ethm}{\end {theorem}}
\newcommand{\bco}{\begin {coro}}
\newcommand{\eco}{\end {coro}}
\newcommand{\bex}{\begin {ex}}
\newcommand{\eex}{\end {ex}}
\newcommand{\be}{\begin {equation}}
\newcommand{\ee}{\end {equation}}
\newcommand{\bp}{\begin {proof}}
\newcommand{\ep}{\end {proof}}
\newcommand{\rt}{\rightarrow}
\newcommand{\ot}{\otimes}
\newcommand{\lb}{\label}

\newcommand{\la}{\lambda}

\newcommand{\stm}{\setminus}

\begin{document}
	
	\title {Types of elements in non-commutative Poisson algebras and Dixmier Conjecture} 
	\begin{tiny}
		\author{ Zhennan Pan, 	Gang Han \thanks{Corresponding author}
			\\School of Mathematics, Zhejiang
			University\\pzn1025@163.com, mathhgg@zju.edu.cn
		}
	\end{tiny}

	\date{March 14, 2025}		
	\maketitle
	
	\begin{abstract}
		Non-commutative Poisson algebras ar	e the algebras having an associative algebra structure
		and a Lie algebra structure together with the Leibniz law.	Let $P$ be a non-commutative Poisson algebra  over  some algebraically closed field of characteristic zero.  For any $z\in P$,  there exist four subalgebras of $P$  associated with the inner derivation $ad_z$ on $P$. Based on the relationships between these four subalgebras, elements of $P$ can be divided into eight types. We will mainly focus on two types of non-commutative Poisson algebras: the usual Poisson algebras and the associative algebras with the commutator as the Poisson bracket. The following problems are studied for such non-commutative Poisson algebras:  how the type of an element changes under homomorphisms between non-commutative Poisson algebras, how the  type of an element  changes after localization, and what the type of the elements of the form $z_1 \ot z_2$ and $z_1 \ot 1 + 1 \ot z_2$  is in the tensor product of non-commutative Poisson algebras $P_1\ot P_2$.  As an application of  above results, one knows that Dixmier Conjecture for $A_1$ holds under certain conditions.  Some properties of the Weyl algebras are also obtained, such as the commutativity of certain subalgebras.
	\end{abstract}

	\tableofcontents	
	
	\section{Introduction}
	\setcounter{equation}{0}\setcounter{theorem}{0}
	\label{chap_int}
	
	In the influential paper \cite{d},	Dixmier conducted an in-depth study of the Weyl algebra $ A_1 $ over fields of characteristic 0.  For any element $z\in A_1$, one has the inner derivation $ad_z$ on $A_1$, and there are four subalgebras related to $ad_z$: $F(z)$, $N(z)$, $D(z)$, and $C(z)$. Based on the relationships between these four subalgebras, Dixmier divided the non-central elements of $A_1$ into five disjoint sets and further studied the properties of the elements belonging to each set. Based on these results, Dixmier determined the generators of the automorphism group of $A_1$. Bavula showed that certain non-commutative domains with GK dimension less than 3 admit an analogous partition as in $A_1$,  and gave many examples \cite{b}.
	
	Dixmier posed six questions at the end of his paper \cite{d}, of which the 3rd, 5th, and 6th have been answered{\cite{J1975}}{\cite{VB2005}}, while the 1st, 2nd, and 4th remain open. Dixmier's first question is: Are all endomorphisms on $ A_1 $ isomorphisms? The similar question for the $n$-th Weyl algebra is referred to as the Dixmier Conjecture. It is known that the Dixmier Conjecture,  the Jacobian Conjecture, and the Poisson Conjecture are all equivalent
	{\cite{HB1982}}{\cite{YT2005}}\cite{av}.\bigskip
	
	Poisson structures appear in a large variety of different contexts, ranging from classical  mechanics,  differential geometry, 	algebraic geometry, and etc. 
	For a systematic study of Poisson algebras, please refer to \cite{cap}. 	 A non-commutative Poisson algebra is an algebra having a not necessarily commutative associative algebra structure
	and a Lie algebra structure, together with the Leibniz law.
	An  associative algebra $A$ can also be regarded as a non-commutative Poisson algebra, with $\{z, w\}= [z,w]=zw-wz$ for $z,w\in A$. To our knowledge,  Kubo's paper \cite{FK1996} was among the earliest works on  non-commutative Poisson algebras. This paper mainly focuses on non-commutative Poisson algebras of Class 1 (which are classical Poisson algebras) and Class 2 (which are associative algebras where the commutator serves as the Poisson bracket).
	
	Given a  non-commutative Poisson algebra $P$,  for any $z\in P$ there are also four subalgebras related to $ad_z$: $F(z)$, $N(z)$, $D(z)$, and $C(z)$, as defined in \cite{d}.
	Based on the relationships between these four subalgebras,  	 	 the elements in  $P$ can be divided into 8 disjoint sets. Thus there are at most 8 types of elements
	in a non-commutative Poisson algebra. The type of an element does not change after applying an automorphism of $P$. Some basic properties of elements of different types are studied in Section 2. 
	
	How the type of an element changes under homomorphisms between non-commutative Poisson algebras is  studied in Section 3. The main result is Theorem \ref{th39}, which will help in the study of homomorphisms between non-commutative Poisson algebras. As a corollary, one knows that if the image of an endomorphism $\varphi$ of the Weyl algebra $A_1$ contains some strictly nilpotent or strictly semisimple element, then  $\varphi$ is an automorphism (Corollary \ref{c311}). 
	
	Next in Section 4, we study how the  type of an element  in a non-commutative Poisson algebra  changes after localization, and obtain some preliminary results.  The torsion algebras of certain elements in the non-commutative Poisson algebra  are determined. See Theorem \ref{p04}.
	
	Finally we study the type of elements in the tensor product of non-commutative Poisson algebras. 	
	Let $P_1$ and $P_2$ be non-commutative Poisson algebras of the same class (Class 1 or Class 2). Then there is a natural Poisson algebra structure on  the tensor product $P_1 \ot P_2$ compatible with those on $P_1$ and $P_2$. Assume that  there are   finite discrete filtrations on $P_1$ and $P_2$ respectively. For an arbitrary element in  $P_1 \ot P_2$, it is hard to determine its type. Therefore we focus on elements in  $P_1 \ot P_2$ of special forms: $z_1 \ot 1 + 1 \ot z_2$ and $z_1 \ot z_2$.
	
	Let $\Theta=z_1 \ot z_2$. Under the assumption that  $gr\ P_1$ and $gr\ P_2$ are commutative under Poisson bracket, one can determine the subalgebras associated with $\Theta $ in terms of those subalgebras associated with $z_i$ for $i=1,2$. Under mild restrictions, one has 
	$F(\Theta)= N(z_1)\otimes N(z_2)$. See Theorem \ref{thm2}. The other cases are also treated.
	Next, we determine  the type of $\Theta$ under various conditions and apply the above results to Weyl algebras. See Section 5.

	Let $ \Gamma = z_1 \ot 1 + 1 \ot z_2$. It is shown that 	$F(\Gamma)= F(z_1)\otimes F(z_2)$ in Theorem \ref{thm1}, which	 does not require the commutativity of $gr\ P_1$ and $gr\ P_2$ under Poisson bracket. We determine the subalgebras associated with  $ \Gamma$ and  the type of $\Gamma$ under various conditions, and apply the above results to Weyl algebras. See Section 6.
	\\[1mm]
	
	Below are some conventions  used in this paper.\\[1mm]
	
	All associative  algebras are assumed to be unital.
	
	$K$ is an algebraically closed field of characteristic zero; $K^{\times}=K \setminus\{0\}$;	 $a, b, c, d, \lambda,\mu$ are elements in $K$.
	
	For a non-commutative Poisson algebra $P$ over $K$, $GK(P)$ denotes the GK dimension of the associative algebra $(P,\ \cdot)$ over $K$.

	$k,l,m,n,s,t$ are integers.
	
	$A_n$ is the $n$-th Weyl algebra over $K$; $p,q$ are the generators of $A_1$ with $[q, p] = 1$.
	
	Given a non-commutative Poisson algebra $P$, $Z(P)$ is the center of the algebra $P$; $P^{*} = P\setminus \{0\}$.
	
	The tensor product "$\ot$" will be always over $K$.
	
	"iff" stands for "if and only if".
	
	\centerline{\textbf{Acknowledgements}}
	
	We would like to heartily thank Yulin Chen and Yangjie Yin for  helpful discussions.

	\section{Types of elements in non-commutative Poisson algebras}\label{chap2}
	\setcounter{equation}{0}\setcounter{theorem}{0}
	
	Let us recall the definition of non-commutative Poisson algebras.
	
	\begin{defi}\label{definition1}
		Let $P$ be a  vector space over some field $F$ endowed with two bilinear operations: a product "$\cdot$" and a  bracket
		$\{\cdot, \cdot\}$,   such that
		
		(1) $(P,\ \cdot)$ is
		an associative unital algebra  over  $F$, which is not necessarily commutative,
		
		(2) $(P, \{\cdot, \cdot \})$ is a Lie algebra  over  $F$, 
		
		(3) (Leibniz rule)
		For any $x, y, z \in P$,
		$	\{x ,y\cdot z\} = \{x, y\} \cdot z+ y\cdot \{x, z\}$.
		
		Then we will call $P$ a non-commutative Poisson algebra.
	\end{defi}

	
	By (3) in the Definition \ref{definition1} and the skew-symmetry of $\{,\}$, one has $$\{y\cdot z,x\} = y\cdot\{z, x\} + \{y,x\}\cdot z,\quad for\ any\ x, y, z \in P.$$
	We will mainly deal with two (large) classes of  non-commutative Poisson algebras in this paper: 	A non-commutative Poisson algebra $P$ such that $(P,\ \cdot)$ is commutative will be referred to as the non-commutative Poisson algebras of  Class 1; a non-commutative associative algebra $A$ such that $\{a,b\}=[a,b]=ab-ba$ for any $a,b\in A$ will be referred to as a non-commutative Poisson algebra of Class 2. Note that there are also non-commutative Poisson algebras such that $\{,\}$ is not the commutator.

	Let $P_1,P_2$ be non-commutative Poisson algebras over $K$. A $K$-linear map $f:P_1\rt P_2$ is called a Poisson homomorphism if $f$ is a homomorphism of associative algebras and is also a homomorphism of Lie algebras. Then the objects of non-commutative Poisson algebras with morphisms being the Poisson homomorphisms form a categoty. Note that the category of Poisson algebras and the category of unital associative algebras are both subcategories of the category of non-commutative  Poisson algebras.
	
	Let $P$ be a  non-commutative Poisson algebra.  The center of
	$(P,\ \cdot)$ and $ (P,\{,\})$ will be denoted by 	$Z(P,\ \cdot)$ and $ Z(P,\{,\})$ respectively. Let $Z(P)=
	Z(P,\ \cdot)\cap Z(P,\{,\})$ be the center of $P$. It is directly verified that $Z(P,\{,\})$ is a Poisson subalgebra of $P$, while $Z(P)$ is an associative subalgebra of 
	$(P,\ \cdot)$. It is easy to see that for non-commutative Poisson algebras of Class 1 or 2, \[Z(P)=Z(P,\{,\}).\]

	In the rest of this section, $P$ will always denote a non-commutative Poisson algebra over $K$ of Class 1 or 2. 
	
	Let \[Der_K(P)=\{f\in End_K(P)\mid\forall z,w\in P, f(zw)=f(z)w+zf(w), f(\{z,w\})=\{f(z),w\}+\{z,f(w)\}.\}\] It is clear that $Der_K(P)$ is a Lie algebra over $K$. For $z \in A$, define 
	$$ad_z:P\rightarrow P,\quad w\mapsto \{z, w\}.$$ 
	Then $ad_z\in Der_K(P)$. One has
	$$ad: P\rightarrow Der_K(P),\quad z\mapsto ad_z,$$
	which is a Lie algebra homomorphism with $\ker ad = Z(P,\{,\})$. We define
	\begin{align}
		&Ev(z) = \big\{\lambda \in K \mid \exists x \in P^{*} \text{ s.t. } ad_z (x) = \lambda x\big\},\notag\\
		&F(z) = \big\{x\in P \mid \dim_KK[ad_z](x)<\infty\big\},\notag\\
		&N(z) = \big\{x\in P \mid \exists m \text{ s.t. } ad_z^m(x) = 0\big\},\notag\\
		&C(z) = \big\{x\in P \mid ad_z(x) =0\big\},\notag\\
		&D(z) = \bigoplus_{\lambda \in Ev(z)} D(z,\lambda),\quad D(z,\lambda):=\big\{ x\in P \mid ad_z(x) = \lambda x\big\}.\notag
	\end{align}
	The set	$Ev(z)$ is an additive sub-monoid of $K$. In particular, $0\in Ev(z)$.

	$F(z)$  also has a similar expression:
	$$F(z) = \bigoplus_{\lambda \in Ev(z)}F(z,\lambda),\quad F(z,\lambda) = \bigcup_{n\geq 1} \ker(ad_z -\lambda)^n.$$
	Thus $N(z) = F(z, 0)$. One has  $F(z) = N(z)$ and $D(z)=C(z)$ iff $Ev(z) = \{0\}$, and $D(z)\supsetneq C(z)$ iff $Ev(z)\supsetneq \{0\}$.
	
	We call $F(z)$ the torsion algebra of $z$, $N(z)$ the nil-algebra of $z$, $D(z)$ the eigenvalue algebra of $z$, $C(z)$ the centralizer of $z$. We will refer to $F(z), N(z),D(z),C(z)$ as the algebras \textbf{associated with} $z$. They satisfy the following relations:
	$$N(z) \subseteq F(z),\quad D(z)\subseteq F(z),\quad C(z) = N(z)\cap D(z) .$$
	The following lemma obviously holds.
	\begin{lemma}\label{lemma0}
		Take $z \in P$. If $F(z) = N(z)$, then $D(z) = C(z)$ and $ Ev(z)=\{0\}$. If $F(z) = D(z)$, then $N(z) = C(z)$.
		\end{lemma}

	\begin{defi}
		Let $P$ be a non-commutative Poisson algebra and $z \in P$.
		
		\noindent1.	If $Ev(z) = \{0\}$, then write
		
		$z\in \Omega_0(P)$, if $C(z) = P$; Clearly, $\Omega_0(P) = Z(P)$;
		
		$z\in \Omega_0'(P)$, if $C(z) = F(z)\subsetneq P$; 
		
		$z\in \Omega_1(P)$, if $C(z) \subsetneq N(z) = P$; 
		
		$z \in \Omega_1'(P)$, if $C(z) \subsetneq N(z) \subsetneq P$;
		
		\noindent2. If $Ev(z) \supsetneq \{0\}$, then write 
		
		$z\in \Omega_2(P)$, if $ D(z) = F(z) = P$;
		
		$z\in \Omega_2'(P)$, if $ D(z) = F(z) \subsetneq P$;

		$z\in \Omega_3(P)$, if $ D(z) \subsetneq F(z) = P$;
		
		$z \in \Omega_3'(P)$, if $ D(z) \subsetneq F(z) \subsetneq P$.

		When there is no ambiguity, we abbreviate $\Omega_i(P)$ and $\Omega'_i(P)$ as  $\Omega_i,\Omega_i', i = 0,1,2,3$.

		Elements in $\Omega_1\ (\text{resp. }\Omega_1',\ \Omega_1\cup \Omega_1') $ are called strictly nilpotent (resp. weakly nilpotent, nilpotent),  elements in $\Omega_2\ (\text{resp. }\Omega_2',\  \Omega_2\cup \Omega_2')$ are called strictly semisimple (resp. weakly semisimple, semisimple), and elements in $\Omega_3\ (\text{resp. }\Omega_3',\ \Omega_3\cup \Omega_3')$ are called strictly Jordan (resp. weakly Jordan, Jordan). 
		
	\end{defi}
	So the elements in $P$ are divided into the  eight disjoint sets above  and each element of $P$ is of one of the eight types above. It is clear that the type of an element does not change after applying an automorphism of $P$.
	
	Let \[\Omega(P) = \Omega_0(P)\cup\Omega_1(P)\cup\Omega_2(P)\cup\Omega_3(P),\quad \Omega'(P) = \Omega'_0(P)\cup\Omega'_1(P)\cup\Omega'_2(P)\cup\Omega'_3(P).\] It is clear that  $\Omega'(P) =P\setminus\Omega(P) $. If $u\in \Omega(P)$ then $F(u) = P$, and if $w\in \Omega'(P)$, then $F(w) \subsetneq P$. Elements in $\Omega(P)$ will be called strict, and elements in $\Omega'(P)$ will be called weak.

	For the Weyl algebra $A_1$, the sets $\Omega_1,\Omega_1',\Omega_2,\Omega_2',\Omega_0'$ are respectively $\Delta_i$ for $i=1,\cdots,5$   in \cite{d}.

	If $z$ is nilpotent, then $Ev(z)=\{0\}$ and $D(z) = C(z)$; if $z$ is semisimple, then $N(z) = C(z)$. The following conclusion holds:
	\begin{lemma}\label{lemma3}
		Let $z\in P, Ev(z)\supsetneq\{0\}$. If $C(z)\subsetneq N(z)$ and $F(z)= P$, then $z\in \Omega_3$.
		If $C(z)\subsetneq N(z)$ and $F(z)\subsetneq P$, then $z\in \Omega_3'$.
		\end{lemma}
	\begin{proof}
		When $C(z)\subsetneq N(z)$, $D(z)\subsetneq N(z)D(z)\subseteq F(z)$, so the above conclusion holds.
	\end{proof}
	We can consider two elements commuting under the bracket in the non-commutative Poisson algebra $P$. Let $z,w\in P,\{z,w\} = 0$, we have $ad_zad_w = ad_wad_z$, and it is direct to prove that  \[F(z)\cap F(w)\subseteq F(z + w ).\] One has the following property.
	\begin{prop}\label{proposition2}
		Let $z,w\in P. $ Assume $\{z,w \}= 0$. If $z\in \Omega_1,w\in \Omega_2$, then $z + w\in \Omega_3$.
		\end{prop}
	\begin{proof}
		We have $P = F(z)\cap F(w) \subseteq F(z + w)\subseteq P$, so equality holds. It is known that $\{z + w,w\} = 0$. For any $\lambda \in Ev(w),v\in D(w,\lambda)$, there exists a sufficiently large integer $n$ such that
		$$(ad_{z + w} - \lambda)^n(v) = (ad_{z + w}-ad_{w})^n(v) = ad_{z}^n(v) = 0.$$
		Thus $Ev(w)\subseteq Ev(z + w),D(w,\lambda)\subseteq F(z + w,\lambda)$. Since
		$$P = \bigoplus_{\lambda \in Ev(w)}D(w,\lambda) = \bigoplus_{\mu\in Ev(z + w)}F(z + w, \mu),$$
		we conclude that $Ev(w) = Ev(z + w),D(w,\lambda)= F(z + w,\lambda)$. 
		
		If for any $ \lambda \in Ev(w), D(w,\lambda) = D(z + w,\lambda)$, then $z = (z + w) - w\in Z(P)$, which is a contradiction. Therefore, there exists $\lambda_0\in Ev(z+w)$ such that $D(z + w,\lambda_0)\subsetneq F(z + w,\lambda_0)$. Hence $D(z + w)\subsetneq F(z + w) = P, z + w\in \Omega_3$.
	\end{proof}
	On the contrary, we can consider how to decompose the elements in $\Omega_3$.
	\begin{prop}
		Let $P$ be a non-commutative Poisson algebra of Class 1 or 2. Assume $ad\ P = Der_K(P)$ and $z\in \Omega_3$. Then there exist $z_1\in \Omega_1, z_2\in \Omega_2,$ such that $\{z_1, z_2\} = 0$ and $z = z_1 + z_2$. If there also exist $z_1'\in \Omega_1, z_2'\in \Omega_2,$ such that $\{z_1', z_2'\} = 0$ and $z = z_1' + z_2'$, then $z_1- z_1' = z_2-z_2'\in Z(P)$.
	\end{prop}
	\begin{proof}
		It is known that $P = F(z) = \bigoplus_{\lambda\in Ev(z)}F(z,\lambda)$. We can define an inner derivation $d_2$ on $P$ as following: $$d_2(v) = \lambda v,\quad v\in F(z,\lambda),\lambda\in Ev(z).$$ Clearly, $d_2\in Der_KP,d_1 = ad_z-d_2\in Der_KP,$ $\{d_1,d_2\}=0$, and $d_1$ is strictly nilpotent, $d_2$ is strictly semisimple. Since $ad\ P = Der_K(P)$, there exist $z_1, z_2$ such that $ad_{z_1} = d_1, ad_{z_2} = d_2$, with  $z_1\in \Omega_1,z_2\in \Omega_2,ad_z = ad_{z_1 + z_2}$. Therefore $z-z_1-z_2=c\in Z(P)$. Then $z_1 +c\in \Omega_1,\{z_1 +c, z_2\}=0,z = (z_1 +c) + z_2$, which satisfies the conditions. 
		
		The proof of the last statement is similar to the proof of Proposition \ref{proposition2} and we omit it.

	\end{proof}
	It is known that if  $A$ is a non-commutative domain (viewed as a non-commutative Poisson algebra of Class 2), and its $GK$ dimension $ GK(A) < 3$, then for any $w \in A$ either $F(w) = N(w)$ or $F(w) = D(w)$. In this case, $\Omega_3 = \Omega_3'= \emptyset $ {\cite{b}}.

	\section{How the type of an element  changes under a Poisson homomorphism}
	\setcounter{equation}{0}\setcounter{theorem}{0}
	
	Let  $P$ be a non-commutative Poisson algebra over $K$.
	For $z\in P$,  $\la \in Ev(z)$ and $k\ge0$, let $F^k(z,\la)=\ker (ad_z-\la )^{k+1}.$ One has
	\[ F(z,\la)=\bigcup_{k\geq 0}F^k(z, \la),\quad F(z)=\bigoplus_{\la\in Ev(z)} F(z,\la). \]
	Then another three sets can be presented as following,
	\[C(z)=F^0(z,0),\quad N(z)=F(z,0),\quad  D(z)=\bigoplus_{\la\in Ev(z)} D(z,\la),  \]
	where $D(z,\la)= F^0(z,\la)$. 
	By the Jacobi identity, $ad_z\{u,w\} = \{u,ad_zw\}+\{ad_zu, w\}, u,w\in P$. For $n\geq 1$, we have
	$$ad_z^n\{u,w\} = \sum_{i = 1}^n\binom{n}{i}\{ad_z^iu,ad_z^{n -i}w\}.$$
	Thus if $u, w\in F(z)$, then $\{u, w\}\in F(z)$. Similarly, $uw\in F(z)$, so $F(z)$ is a non-commutative Poisson subalgebra of $P$. $N(z), D(z), C(z)$ are also non-commutative  Poisson subalgebras of $P$, which can be proved in the same way. 
	
	Let $m:P\ot P\rt P, a\otimes b\mapsto ab$ be the multiplication. For
	\(\phi\in Der_K(P), \lambda,\mu\in K,\) one has \[ (\phi-\lambda-\mu)\cdot m=m\cdot[(\phi-\lambda)\otimes 1+1\otimes(\phi-\mu)],\] it follows that for $a,b\in P$,
	\be \left(\phi-\lambda-\mu\right)^{n}\left(ab\right)= \sum_{i=0} ^{n}  \binom{n}{i} \left(\phi-\lambda\right)^{i}\left(a\right)  \cdot \left(\phi-\mu\right)^{n-i}\left(b\right) .\label{eq0}
	\ee
	Let $F^{-1}(z,\la)=\{0\}$. By the abvoe, one has 
	\ble\label{lemma4}
	Let $z\in P,\la,\mu \in Ev(z,P),k,l\ge 0$, then $m$ induces
	\be \lb{32}F^k(z,\la) \times F^l(z,\mu)\rt F^{k+l}(z,\la+\mu), \ee and
	\be\lb{33} [F^k(z,\la)\stm F^{k-1}(z,\la)] \times [F^l(z,\mu)\stm F^{l-1}(z,\mu)]\rt F^{k+l}(z,\la+\mu)\stm F^{k+l-1}(z,\la+\mu). \ee
	From the above, one has
	\begin{align} \lb{34} [F^k(z,\la)/F^{k-1}(z,\la)] &\times [F^l(z,\mu)/ F^{l-1}(z,\mu)]\rt F^{k+l}(z,\la+\mu)/F^{k+l-1}(z,\la+\mu),\\
		(a+F^{k-1}(z,\la)&, b+F^{l-1}(z,\mu))\mapsto ab+F^{k+l-1}(z,\la+\mu).\notag \end{align}
	
	Let $\mu = \la$, we have
	\[ (F^1(z,\la))^n\subseteq F^n(z,n\la),\quad[F^1(z,\la)\stm F^0(z,\la)]^n\subseteq F^n(z,n\la)\stm F^{n-1}(z,n\la). \]
	
	\ele
	\bp
	For $u\in F^k(z,\la),w\in F^l(z,\mu)$, by equation (\ref{eq0}) one has $(ad_z - \la -\mu)^{k + l + 1}(uw)= 0$, thus $uw \in F^{k + l}(z, \la + \mu)$ and equation (\ref{32}) holds. It follows from (\ref{32}) that (\ref{34}) is well-defined.
	
	If  $u\in F^k(z,\la)\stm F^{k-1}(z,\la),w\in F^l(z,\mu)\stm F^{l-1}(z,\mu),$ then
	$(ad_z-\la)^k(u)\neq 0,(ad_z-\mu)^l(w)$
	
	\noindent$\neq 0.$
	By equation (\ref{eq0}), one has 
	\[(ad_z - \la -\mu)^{k + l} = \binom{k + l}{k}(ad_z - \la)^k(u)(ad_z-\mu)^l(w)\neq 0.\]
	Thus $uw \in F^{k+l}(z,\la+\mu)\stm F^{k+l-1}(z,\la+\mu).$ Thus (\ref{33}) holds. The last statement follows from (\ref{32}) and (\ref{33}).
	\ep
	For $i\geq 0$, let $F^i(z) = \oplus_{\lambda\in Ev(z)}F^i(z, \lambda).$ Then by Lemma \ref{lemma4}, $\{F^i(z)\}_{i\geq 0}$ is a filtration of $(F(z),\ \cdot)$.  
	Similar discussions for associative algebras can be found in Section 6 of \cite{d} and Section 2  of \cite{b}. 
	
	\begin{defi}
		Let $A$ be an associative algebra, and $B$ be a subalgebra of $A$, $ B[X]=\{\sum_{i=0}^n b_i X^i\mid b_i\in B, n\ge0\}$. Let $z\in A$. If for any $ f(X)\in B[X]\stm \{0\}$, $ f(z) \neq 0$, we say $z$ is right algebraically independent over $B$; otherwise $z$ is right algebraically dependent over $B$. Analogously one can define left algebraic independence (resp. left algebraic dependence) over $B$.
	\end{defi}
	In this section,  algebraic independence will refer to right algebraic independence.
	
	\begin{prop}\lb{p23}
		Let $A$ be a domain and $GK(A) < \infty$. Let $B$ be a subalgebra of $A$. If $\exists z\in A$, such that $z$ is (right) algebraically independent over $B$, then $GK(A) \geq GK(B) + 1$. 
	\end{prop}
	\bp
	First, notice that $\sum_{k = 0}^{\infty} B \cdot z^k$ is a direct sum as $z$ is algebraically independent over $B$.
	Let $W$ be a subspace of $B$. Assume that $1\in W$ and $dim\ W < \infty$.  Let $V:= W \bigoplus K z$,
	then $$V^{2n} = (W\bigoplus Kz)^{2n} \supseteq W^n\bigoplus W^nz \bigoplus \cdots \bigoplus W^nz^n.$$ As  $A$ is a domain, $dim\ W^nz^i = dim\ W^n$ for $i =1,\cdots,n$, hence $dim\ V^{2n} \geq n\cdot dim\ W^n$. One has
	\[\frac{log\ dim\ V^{2n}}{log\ 2n} \geq \frac{log\ (n\cdot dim\ W^n)}{log\ 2n} = \frac{log\ n + log\ dim\ W^n}{log\  2 + log\ n},\]
	\[\varlimsup_{n \to \infty}\frac{log\ dim\ V^{2n}}{log\ 2n}\geq \varlimsup_{n\to \infty}\frac{log\ n + log\ dim\ W^n}{log\ 2 + log\ n} = 1 +  \varlimsup_{n\to \infty}\frac{ log\ dim\ W^n}{ log\ n}.\]
	Taking the supremum over all finite dimensional subspace $W$ of $B$,
	one has  $GK(A) \geq GK(B) + 1$.
	\ep
	We have the following corollary.
	\begin{coro}
		Let $B$ be a subalgebra of the domain $A$ with $GK(A) < GK(B) + 1< \infty$, then  $z$ is (right) algebraically dependent over $B$ for all $ z\in A$.
	\end{coro}
	
	From  now until the end of this section, \textbf{$P$ is assumed to be a non-commutative Poisson algebra and  $(P,\ \cdot)$ is a domain.}
	\ble\lb{25}
	Let $ z\in P$. If $w \in D(z,\lambda)\backslash \{0\}$ for some $\lambda \in Ev(z)\backslash \{0\}$, then $w$ is (right) algebraically independent over $N(z)$.
	\ele
	\bp
	Assume $f(X) = a_nX^n + a_{n-1}X^{n-1} + \cdots + a_0$, $a_i \in F(z,0)$, $ n \geq 1, a_n \neq 0$, 
	then $f(w) = a_nw^n + a_{n-1}w^{n-1} + \cdots + a_0$. For $i = 0, 1, \cdots , n$, $a_iw^i\in F(z,i\lambda)$,
	thus they are linearly independent. If $f(w) = 0$, then $a_iw^i=0$ for all $i$. As $(P,\ \cdot)$ is a domain, $a_i = 0$ for all $i$. In particular, $a_n = 0$,  which contradicts our assumption.
	\ep
	
	\ble\lb{26}
	Let  $ z\in P$. Assume  $w\in F(z,\lambda)\backslash D(z,\lambda)$ for some $\lambda \in Ev(z)$, then $w$ is  (right) algebraically independent over $D(z)$.
	\ele
	\bp
	Assume $f(X) = a_nX^n + a_{n-1}X^{n-1} + \cdots + a_0\in D(z)[X]\backslash\{0\}$, $n \geq 1, a_n \neq 0$. Let $w\in F^k(z,\lambda)\backslash F^{k-1}(z,\lambda), k\geq 1$, then 
	$f(w) = a_nw^n + a_{n-1}w^{n-1} + \cdots + a_0.$
	As $a_nw^n\in F^{nk}(z)\backslash  F^{nk-1}(z)$, $a_jw^j\in F^{jk}(z)$ for $j<n$ and $\{F^i(z)\}_{i\geq 0}$ is a filtration of $F(z)$, $f(w)\ne0$. So  $w$ is  (right) algebraically independent over $D(z)$.

	\ep

	\begin{prop}
		Assume $x,y \in P$ and  $C(x) = C(y)$. Let $N^k(x) = F^k(x, 0).$ Similar for $N^k(y)$. Then  $  N^k(x) = N^k(y), \forall k \geq 0$ and $N(x) = N(y)$.
	\end{prop}
	\bp
	We will prove it by induction. Assume $k>0$ and $N^i(x) = N^i(y), 0\leq i \leq k-1$. Then we have the following commutative diagram:
	
	\[\begin{tikzcd}
		\cdots  \arrow[r,"ad_x"]& N^k(x)  \arrow[r,"ad_x"] & N^{k-1}(x) \arrow[r,"ad_x"] \arrow[d,equal] & \cdots  \arrow[r,"ad_x"] & N^0(x) \arrow[r,"ad_x"] \arrow[d,equal] & 0\\
		\cdots  \arrow[r,"ad_y"]& N^k(y)  \arrow[r,"ad_y"] & N^{k-1}(y) \arrow[r,"ad_y"]  & \cdots  \arrow[r,"ad_y"] & N^0(y) \arrow[r,"ad_y"]  & 0
	\end{tikzcd}\]
	
	For $z \in N^k(x)$, $ad_x(z) \in N^{k-1}(y) $, thus 
	$ ad_y^k(ad_x(z)) = 0$. By assumption  $C(x) = C(y)$, one has $ \{x,y\} = 0$ and $\{ad_x,ad_y\} = ad_{\{x,y\}} =0$.
	Hence $ad_y^k(z)\in N^0(x) = N^0(y)$, thus  $ad_y^{k+1}(z)=0$ and $z\in N^k(y)$. 
	\ep

	Let $P,P_1$ be non-commutative Poisson algebras and let \( \varphi: P \to P_1 \) be a Poisson homomorphism.  Let \( w \in P \). The following inclusions obviously hold:
	\[
	\varphi(F(w)) \subseteq F(\varphi(w)), \quad \varphi(N(w)) \subseteq N(\varphi(w)),
	\]
	\[
	\varphi(D(w)) \subseteq D(\varphi(w)), \quad \varphi(C(w)) \subseteq C(\varphi(w)).
	\]
	If \( \varphi \) is an isomorphism, then equality holds in all these inclusions, and the type of the element $w$ does not change after applying $\varphi$. If \( \varphi \) is not an isomorphism, then the type of $w$ may change  after applying $\varphi$.
	
	\begin{theorem}\label{th39}
		Let $P,P_1 $ be non-commutative Poisson algebras. Assume $(P,\  \cdot), (P_1,\ \cdot)$ both are domains.	If $\phi: P\to P_1 $ is an injective homomorphism and  $GK(P_1 ) < GK(P) + 1< \infty$,  then we have the following:
		
		(1) $\phi(\Omega_0(P)) \subseteq \Omega_0(P_1 ) \bigcup \Omega_0'(P_1 )$.
		
		(2)   $\phi(\Omega_1(P)) \subseteq \Omega_1(P_1 ) \bigcup \Omega_1'(P_1 )$.   
		
		(3)  $\phi(\Omega_2(P)) \subseteq \Omega_2(P_1 ) \bigcup \Omega_2'(P_1 )$.
		
		(4)  $\phi(\Omega_3(P)) \subseteq \Omega_3(P_1 ) \bigcup \Omega_3'(P_1 )$.
		
		(5)   $\phi(\Omega_3'(P)) \subseteq \Omega_3'(P_1 )$.
		
		(6)    $\phi(\Omega'(P)) \subseteq \Omega'(P_1 )$, i.e. $\phi(\Omega'(P))\bigcap (\Omega(P_1 )) = \emptyset$.

		Assume that $z\in P$ satisfies $GK(F(z,P)) = GK(P)$ in addition. Then
		
		(7)  If $z\in \Omega_1'(P)$, then $\phi(z) \in \Omega_1'(P_1 )$.
		
		(8)  If $z\in \Omega_2'(P)$, then $\phi(z) \in \Omega_2'(P_1 )$.

	\end{theorem}
	\bp
	Let $z\in P$ and $z_1 = \phi(z)$. For $B\subseteq P_1$ a subalgebra, $C(z_1,B)=\{w\in B\mid ad_z(w)=0\}$ denotes the centralizer of $z_1$ in $B$, and analogous notation will be used. As $\phi$ is an injective homomorphism, one has 
	\[\phi(C(z,P))=C(z_1,\phi(P)),\quad \phi(D(z,P))=D(z_1,\phi(P)),\]
	\[\phi(N(z,P))=N(z_1,\phi(P)),\quad \phi(F(z,P))=F(z_1,\phi(P)).\]
	
	For $B\subseteq P_1$, if $GK(B)\geq GK(P),$ then for any $ w\in P_1,$ $w$ is algebraically dependent over $B$. Otherwise by Proposition \ref{p23}, $GK(P_1)\geq GK(B) + 1 \geq GK(P) + 1$, contradict. Furthermore, if $GK(N(z_1,P_1))\geq GK(P)$, then $Ev(z_1) = \{0\}, D(z_1,P_1) = C(z_1,P_1)$ by Lemma \ref{25}, and if $GK(D(z_1,P_1))\geq GK(P)$, then $F(z_1, P_1) = D(z_1, P_1)$ by Lemma \ref{26}.
	
	(1) Assume $z\in \Omega_0(P)$. Then \[ C(z_1,P_1 ) \supseteq  C(z_1,\phi(P))=\phi(P), \] 
	Thus $GK(N(z_1,P_1))\geq GK(P), GK(D(z_1,P_1))\geq GK(P)$ and $F(z_1,P_1)= D(z_1,P_1)=C(z_1,P_1)$. 
	Thus $z_1\in \Omega_0(P_1)\cup \Omega_0'(P_1)$.

	(2) Assume $z\in \Omega_1(P)$, then $C(z,P)\subsetneq N(z,P)=P$, and 
	\[C(z_1,\phi(P))\subsetneq N(z_1,\phi(P))=\phi(P)\subseteq N(z_1,P_1).\]
	Thus $GK(N(z_1,P_1))\geq GK(P)$ and $Ev(z_1) = \{0\}$. 
	Thus 
	$ z_1\in \Omega_1(P_1 )\cup\Omega_1'(P_1 )$.

	(3) Similar to (2).

	(4) Assume $z\in \Omega_3(P)$. Then $C(z,P)\subsetneq D(z,P)\subsetneq F(z, P)=P$, and 
	\[C(z_1,\phi(P))\subsetneq D(z_1,\phi(P))\subsetneq F(z_1,\phi(P)).\]
	Thus $C(z_1,P_1)\subsetneq D(z_1,P_1)\subsetneq F(z_1,P_1)$ and $z_1\in \Omega_3(P_1)\bigcup \Omega_3'(P_1).$
	
	(5) Similar to (4).

	(6) 
	Assume $z\in \Omega'(P),$ then $F(z, P)\subsetneq P,F(z_1,\phi(P))\subsetneq \phi(P),F(z_1,P_1)\subsetneq P_1,$ thus $\varphi(\Omega'(P))\subseteq \Omega'(P_1)$, $\phi(\Omega'(P))\bigcap (\Omega(P_1 )) = \emptyset$.

	(7) Assume $z\in \Omega_1'(P)$. Then $C(z,P)\subsetneq N(z,P) = F(z, P)\subsetneq P$, and 
	\[C(z_1,\phi(P))\subsetneq N(z_1,\phi(P))\subsetneq \phi(P),\]
	thus $C(z_1,P_1)\subsetneq N(z_1,P_1)\subsetneq P_1.$ As $GK(F(z, P)) = GK(P)$, we have
	\[GK(N(z_1,P_1 ))\geq GK(N(z_1,\phi(P))) =GK(N(z, P)) = GK(F(z, P)) =GK(P).\]
	Thus  $Ev(z_1)=\{0\}$ and  $z_1\in \Omega_1'(P_1)$.
	

	(8) Similar to (7).

	\ep
	We will apply above result to \( A_1 \). As $GK(A_1) = 2$, $\Omega_3(A_1) = \Omega_3'(A_1) = \emptyset$. Let
	\[
	\Lambda = \{ z \in A_1 \mid \exists w \in A_1, [z, w] = 1 \}.
	\]	
	For any \( z \in \Lambda \), one has \( C(z) = K[z] \) by {\cite{G2012}}. Additionally, one has \( \Lambda \cap K[p] = \{ ap + b \mid a \in K^{\times}, b \in K \} \) (see {\cite{BO2024}}).
	
	\begin{prop}\label{lemma1}
		Let \( \varphi: A_1 \to A_1 \) be a homomorphism. Assume \( \text{Im} \, \varphi \subsetneq A_1 \). Then:
		
		\begin{enumerate}
			\item \( \varphi(A_1) \cap \Omega_1 = \emptyset \), and \( \varphi(\Omega_1 \cup \Omega_1') \subseteq \Omega_1' \).
			\item \( \varphi(A_1) \cap \Omega_2 = \emptyset \), and \( \varphi(\Omega_2 \cup \Omega_2') \subseteq \Omega_2' \).
			\item \( \varphi(\Omega_0) = \Omega_0 \), and \( \varphi(\Omega_0') \subseteq \Omega_1' \cup \Omega_2' \cup \Omega_0' \).
		\end{enumerate}
	\end{prop}
	
	\begin{proof}
		Since $A_1$ is simple, $\varphi$ is injective. Take $w\in A_1$. If $w$ is nilpotent or semisimple, then $GK(F(w)) = GK(A_1) = 2$. Thus by Theorem \ref{th39}, we have $\varphi(\Omega_i)\subseteq \Omega_i\cup \Omega_i'$ and $\varphi(\Omega_i')\subseteq \Omega_i'$ for $i = 1,2.$
		
		\textbf{1.} Suppose \( \varphi(w) = z \in \Omega_1 \). Then , from the above, we know that \( w \in \Omega_1 \). There exist automorphisms \( \alpha_1, \alpha_2 \in \text{Aut}_K(A_1) \) such that
		\[
		\alpha_1(w) = f(p), \quad \alpha_2(z) = g(p), \quad f, g \in K[X], \quad \deg f, \deg g \geq 1.
		\]
		Thus, there exists an automorphism \( \widetilde{\varphi}: A_1 \to A_1 \) making the following diagram commute:
		
		\[
		\begin{tikzcd}
			A_1 \arrow[r, "\varphi"] \arrow[d, "\alpha_1"] & A_1 \arrow[d, "\alpha_2"] \\
			A_1 \arrow[r, "\widetilde{\varphi}"] & A_1
		\end{tikzcd}
		\quad
		\begin{tikzcd}
			w \arrow[r, "\varphi"] \arrow[d, "\alpha_1"] & z \arrow[d, "\alpha_2"] \\
			f(p) \arrow[r, "\widetilde{\varphi}"] & g(p)
		\end{tikzcd}
		\]	
		
		\vspace{0.7cm}
		Since \( p \in C(f(p)) \), we have \( \widetilde{\varphi}(p) \in C(g(p)) = K[p] \). Moreover, \( [\widetilde{\varphi}(q), \widetilde{\varphi}(p)] = 1 \), so \( \widetilde{\varphi}(p) \in \Lambda \). Therefore,	
		\[
		\widetilde{\varphi}(p) = ap + b \in \Lambda \cap K[p], \quad a \in K^{\times}, b \in K.
		\]		
		From \( [a^{-1}q, ap + b] = 1 \), we conclude that \( \widetilde{\varphi}(q) - a^{-1}q \in C(\widetilde{\varphi}(p)) = K[p] \). Hence,		
		\[
		\widetilde{\varphi}(q) = a^{-1}q + h(p), \quad h \in K[X].
		\]		
		Thus, \( \widetilde{\varphi}(A_1) = A_1 \), which leads to a contradiction.
		
		\textbf{2.} Let \( h = pq \). Suppose \( \varphi(w) = z \in \Omega_2 \). Then \( w \in \Omega_2 \), and there exist automorphisms \( \alpha_1, \alpha_2 \in \text{Aut}_K(A_1) \) such that	
		\[
		\alpha_1(w) = ch + d, \quad \alpha_2(z) = ah + b, \quad a, c \in K^\times, \quad b, d \in K.
		\]	
		Thus, there exists an automorphism \( \widetilde{\varphi}: A_1 \to A_1 \) making the following diagram commute:
		
		\[
		\begin{tikzcd}
			A_1 \arrow[r, "\varphi"] \arrow[d, "\alpha_1"] & A_1 \arrow[d, "\alpha_2"] \\
			A_1 \arrow[r, "\widetilde{\varphi}"] & A_1
		\end{tikzcd}
		\quad
		\begin{tikzcd}
			w \arrow[r, "\varphi"] \arrow[d, "\alpha_1"] & z \arrow[d, "\alpha_2"] \\
			ch+d \arrow[r, "\widetilde{\varphi}"] & ah+b
		\end{tikzcd}
		\]
		
		\vspace{0.7cm}
		Let \( c_1 = a^{-1}c, d_1 = a^{-1}d - a^{-1}b, w_1 = c_1h + d_1 \). Then \( \widetilde{\varphi}(w_1) = h \), and \( [h, \widetilde{\varphi}(p)] = \widetilde{\varphi}([w_1, p]) = c_1 \widetilde{\varphi}(p) \). Since \( [\widetilde{\varphi}(q), \widetilde{\varphi}(p)] = 1 \), we conclude that, \( \widetilde{\varphi}(p) \in D(h, c_1) \cap \Lambda \), and \( c_1 \in \text{Ev}(pq) = \mathbb{Z} \). From \cite{BO2024} we have 
		\[
		\Lambda \cap \left( \bigcup_{n \in \mathbb{Z}} D(h, n) \right) = \{ \lambda p, \lambda q \mid \lambda \in K^\times \}.
		\]
		Suppose \( \widetilde{\varphi}(p) = \lambda_0 p \), where \( \lambda_0 \in K^\times \). Then \( \widetilde{\varphi}(q) = \lambda_0^{-1}q + h(p),  h \in K[X] \). Therefore, \( \widetilde{\varphi}(A_1) = A_1 \), which leads to a contradiction.
		
		\textbf{3.} Clearly, \( \varphi(\Omega_0) = \Omega_0 = K \). If \( w \in \Omega_0' \), then \( F(w) \subsetneq A_1 \) and \( F(\varphi(w)) \subsetneq A_1 \), so \( \varphi(\Omega_0') \subseteq \Omega_1' \cup \Omega_2' \cup \Omega_0' \).
	\end{proof}
	Then one has the following.
	\begin{coro}\lb{c311}
		Let \( \varphi: A_1 \to A_1 \) be a homomorphism. Assume \( \text{Im} \, \varphi \cap (\Omega_1\cup \Omega_2)\ne \emptyset \). Then $\varphi\in Aut(A_1)$.
	\end{coro}
	
	\section{The torsion algebra of a strict element after localization}
	\setcounter{equation}{0}\setcounter{theorem}{0}
	Throughout this section, $P$ is a non-commutative Poisson algebra over $K$ with $ (P,\ \cdot)$ being a right Ore domain. Let $\widetilde{P}$ be its right division ring.  Assume that there is a bracket $\{,\}$ on $\widetilde{P}$  such that  $\widetilde{P}$ is a   non-commutative Poisson algebra, and the natural inclusion $P\hookrightarrow \widetilde{P}$ is a Poisson homomorphism. It is clear that if the bracket exists then it must be unique. (The localization of $P$ with respect to an arbitrary right Ore set can be treated similarly and will be omitted.)
	
	For $0\ne s,t\in P$, we have
	$0=\{s, tt^{-1}\} =t\{s,t^{-1}\} + \{s,t\}t^{-1},$
	thus $$\{s, t^{-1}\} = -t^{-1}\{s,t\}t^{-1}.$$ Similarly, $\{s^{-1}, t\} = -s^{-1}\{s, t\}s^{-1}.$ Furthermore, we have
	\begin{align}
		0=\{s^{-1},tt^{-1}\} &= t\{s^{-1}, t^{-1}\} + \{s^{-1}, t\}t^{-1}\notag\\
		&= t\{s^{-1}, t^{-1}\} -s^{-1}\{s,t\}s^{-1}t^{-1}\notag
		.\notag 
	\end{align}
	Thus $$\{s^{-1}, t^{-1}\} = t^{-1}s^{-1}\{s,t\}s^{-1}t^{-1}.$$ Similarly, by $\{ss^{-1}, t^{-1}\} = 0$ one has $$\{s^{-1}, t^{-1}\} = s^{-1}t^{-1}\{s,t\}t^{-1}s^{-1}.$$ So
	\begin{equation}
		t^{-1}s^{-1}\{s,t\}s^{-1}t^{-1} = s^{-1}t^{-1}\{s,t\}t^{-1}s^{-1}.
		\label{eq4}\end{equation}
	The equation (\ref{eq4}) may not hold for arbitrary $P$, which means that the Poisson structure on  $\widetilde{P}$ may not exist. But it does exist for $P$ of Class 1 or Class 2, as we will see soon.

	Assume $P$ is of Class 1. In this case,  (\ref{eq4}) clearly holds. To be compatible with the bracket on $P$ and to obey Leibniz rule,  for $as^{-1} ,bt^{-1}\in \widetilde{P}$ one must have
	\[\{as^{-1},bt^{-1}\}=(-as^{-1}\{s,b\}+abt^{-1}s^{-1}\{s,t\}+\{a,b\}-bt^{-1}\{a,t\})(st)^{-1}.\]
	Then it is direct to verify that the bracket is well-defined and that  $\widetilde{P}$ is also a non-commutative Poisson algebra of Class 1, and the natural inclusion $P\hookrightarrow \widetilde{P}$ is a Poisson homomorphism. See \cite{cap}.
	
	Assume $P$ is of Class 2. Then both sides of (\ref{eq4}) equal $s^{-1}t^{-1}-t^{-1}s^{-1}$, thus (\ref{eq4}) holds. It is easy to see that the bracket
	\[\{as^{-1},bt^{-1}\}=as^{-1}\cdot bt^{-1}-bt^{-1}\cdot as^{-1},\quad as^{-1} ,bt^{-1}\in \widetilde{P}, \] makes $\widetilde{P}$ into a  non-commutative Poisson algebra of Class 2 such that the natural inclusion $P\hookrightarrow  \widetilde{P}$ is a Poisson homomorphism.

	Identify $P$  with the Poisson subalgebra of $ \widetilde{P}$. In this section we will determine the torsion algebra in  $\widetilde{P}$ of a strict element $x\in P$, extending some results in \cite{d}.
	
	Recall that $P$ is a  non-commutative Poisson algebra over $K$ such that $ (P,\ \cdot)$ is a right Ore domain. For $x\in P, \mu\in Ev(x, \widetilde{P})$ and $E\subseteq \widetilde{P}$ being a subspace satisfying $ad_x(E)\subseteq E$, let $F(x, E, \mu)=F(x,\widetilde{P} , \mu)\cap E$ and  $D(x, E, \mu)=D(x,\widetilde{P} , \mu)\cap E$.
	Analogous notations  will be used throughout.
	
	\begin{lemma}
		Assume	$x \in P$ is strictly semisimple, and  $\{x, z\} = \lambda z, \lambda \in K$ and $0 \neq z \in \widetilde{P}$. Then $z = uw^{-1}$ where $u,w$ are eigenvectors of $ad_x$ in $P$.
	\end{lemma}
	\begin{proof}
		Let	$I = I(z) := \{ y\in P \mid zy \in P\},$ which  is clearly a nonzero right ideal of $(P,\ \cdot)$. Let $y\in I$. one has
		\[z\cdot\{x,y\} =\{x,z y\}-\{x,z\}\cdot y =\{x,z y\}-\lambda\cdot zy\in P.\]
		Thus $ad_x(I) \subseteq I.$	
		As $x$ is strictly semislmple in $P$,  $ I = \oplus_{\mu \in Ev(x, P)}D(x, I, \mu)$. Then there exists $\mu \in Ev(x, P)$ and $w \in I_{\mu}\stm\{0\}$ such that $ zw = u \in D(x, P,\lambda + \mu).$ So $ z = uw^{-1}$. 
	\end{proof}
	\begin{coro}
		Assume $x\in P$ is strictly semisimple, then 
		$Ev(x, \widetilde{P}) = \{\lambda - \mu \mid\lambda, \mu \in Ev(x, P)\}.$
		For $\nu \in Ev(x, \widetilde{P})$, 
		\[D(x, \widetilde{P}, \nu)=\{uw^{-1} \mid  u \in D(x, P, \mu_0 + \nu), w \in D(x, P, \mu_0)\}\ and\  \mu_0\in Ev(x,P).\}\]
		Meanwhile, 
		$D(x, \widetilde{P}) = \oplus_{\nu \in Ev(x, \widetilde{P})}D(x, \widetilde{P}, \nu).$
	\end{coro}
	\begin{theorem}\lb{p04}
		Assume $x \in P$ satisifies $F(x,P) = P$, then $F(x, \widetilde{P}) = PC^{-1}$ where $$C = \big( \bigcup_{\lambda \in Ev(x, P)}D(x, P, \lambda)\big)  \backslash \{0\}.$$
	\end{theorem}
	\begin{proof}
		Assume $w \in C,$ then $w \in D(x, P, \lambda)\backslash \{0\} $ for some $\lambda\in Ev(x,P)$ and 
		\[w^{-1} \in D(x, \widetilde{P}, -\lambda)\backslash \{0\} \subseteq F(x, \widetilde{P})\backslash \{0\}.\]
		Assume $u \in P$. As $P = F(x, P) \subseteq F(x, \widetilde{P}),$  $uw^{-1} \in F(x, \widetilde{P})$ and one has  $ PC^{-1} \subseteq F(x, \widetilde{P}).$
		
		Assume $z \in F(x, \widetilde{P})$, then there exists a subspace $E \subseteq F(x, \widetilde{P})$ such that $z \in E, ad_x(E) \subseteq E, \dim_K E < \infty$. Let $I = I(E) := \{w \in P \mid ew \in P, \forall e \in E\}  $, which is a nonzero right ideal of $(P,\ \cdot)$. For  $e \in E, w \in I$, by the definition of $I$ and $E$, one has
		\[	e\{x,w\}=\{x,ew\}-\{x,e\}w\in P,\]
		thus $\{x, w\}\in I,ad_x(I) \subseteq I.$ 
		As $F(x, P) = P$, we have 
		\[I = \bigoplus_{\mu \in Ev(x, P)} F(x, I, \mu).\] 
		So there exists $\mu \in Ev(x, P)$ and $ 0 \neq w \in  D(x, I, \mu)$ satisfying $ zw = u \in P,$ thus $ z = uw^{-1} \in PC^{-1}$. So one also has $F(x, \widetilde{P}) \subseteq PC^{-1}$, and the equality holds. 
	\end{proof}
	\begin{coro}\lb{p03}
		Assume $x \in P$ is strictly semisimple, then $F(x, \widetilde{P}) = D(x, \widetilde{P}) = PC^{-1}$, where $C = \left( \bigcup_{\lambda \in Ev(x, P)}D(x, P, \lambda)\right)  \backslash \{0\}$.
	\end{coro}
	\begin{proof}
		Assume $u \in P, w \in C.$ As $x$ is strictly semisimple, $u = \sum_{i}u_i, u_i \in D(x, P, \lambda_i)$. One has $u_iw^{-1} \in D(x, \widetilde{P}),$ thus $uw^{-1} \in D(x, \widetilde{P})$, $PC^{-1} \subseteq D(x, \widetilde{P}).$
		
		On the other hand, as $F(x,P) = P$, by Theorem \ref{p04}	
		\[F(x, \widetilde{P}) = PC^{-1} \subseteq D(x, \widetilde{P}) \subseteq F(x, \widetilde{P}),\] 
		thus the equalities hold. 
	\end{proof}
	For example, the strictly semisimple element  $pq\in A_1$ is weakly semisimple in $\widetilde{A_1}$.
	
	\begin{coro}\lb{c04}
		Assume $x \in P $ is strictly nilpotent, then $F(x, \widetilde{P}) = N(x, \widetilde{P}) = PC^{-1},$ where $ C =C(x, P) \backslash \{0\}.$
	\end{coro}
	\begin{proof}
		Assume $u \in P, w \in C, $ then $ ad_x(uw^{-1}) = ad_x(u)w^{-1}.$ As $x$ is strictly nilpltent, $\exists m > 0$ such that $ad_x^m(u) = 0$. Thus $ad_x^m(uw^{-1}) = 0. $ So $PC^{-1} \subseteq N(x, \widetilde{P}).$

		On the other hand, as $F(x,P) = P$ and $ Ev(P) = \{0\}$, by Theorem \ref{p04}	
		\[F(x, \widetilde{P}) = PC^{-1} \subseteq N(x, \widetilde{P}) \subseteq F(x, \widetilde{P}),\] 
		thus the equalities hold.
	\end{proof}
	For example, the strictly nilpotent element  $p\in A_1$ is weakly nilpotent in $\widetilde{A_1}$. The following result can be proved analogously.
	
	\begin{coro}\lb{p05}
		Assume $x \in P$ is of strictly Jordan type $(x \in \Omega_3(P)),$  then $F(x, \widetilde{P}) = PC^{-1}$ with $C = \left( \bigcup_{\lambda \in Ev(x, P)}D(x, P, \lambda)\right)  \backslash \{0\}$.
	\end{coro}

	\section{Tensor Product of non-commutative Poisson algebras and classification of elements of the form $ z_1 \ot z_2 $}
	\setcounter{equation}{0}\setcounter{theorem}{0}
	Let $P_1,P_2$ be two non-commutative Poisson algebras over $K$. Let $P_1\ot P_2$ be their tensor product, which is an associative algebra. Let $\phi_i:P_i\rt P_1\ot P_2$ be the natural inclusion for $i=1,2$.  Assume that there exist a bracket  $\{,\}$ on $P_1\ot P_2$  such that  $P_1\ot P_2$ is a   non-commutative Poisson algebra, and the natural inclusions $\phi_i$ are Poisson homomorphisms for $i=1,2$ with  $\{\phi_1(P_1), \phi_2(P_2)\}=0 $. It is clear that if the bracket exists then it must be unique.
	
	For $a_1\ot a_2,b_1\ot b_2\in P_1\ot P_2$, one has 
	\begin{align}\lb{51}
		\{ a_1\ot a_2,b_1\ot b_2\} &= \{ a_1\ot a_2,b_1\ot 1\cdot 1\ot b_2\}\notag\\&=b_1\ot 1\cdot \{ a_1\ot a_2,1\ot b_2\}+\{ a_1\ot a_2,b_1\ot 1\}\cdot 1\ot b_2\notag\\
		&= b_1a_1\ot \{a_2,b_2\}+\{a_1,b_1\}\ot a_2b_2.
	\end{align}
	Similarly one has 	\begin{align}\lb{52}
		\{ a_1\ot a_2,b_1\ot b_2\} &= \{ a_1\ot a_2,1\ot b_2\cdot b_1\ot 1\}\notag\\
		&= \{a_1,b_1\}\ot b_2a_2+a_1b_1\ot \{a_2,b_2\}.
	\end{align}
	By the equality of (\ref{51}) and  (\ref{52}) one has 
	\begin{equation}\lb{53}
		\{a_1,b_1\}\ot[a_2,b_2]=[a_1,b_1]\ot \{a_2,b_2\},
	\end{equation}
	which does not hold for arbitrary non-commutative Poisson algebras $P_1,P_2$. 
	
	If $P_1,P_2$ are both non-commutative Poisson algebras of Class 1, then (\ref{53}) holds. It is easy to verify that the bracket on $P_1\ot P_2$ defined by
	\[\{a_1\ot a_2,b_1\ot b_2\}=\{a_1,b_1\}\ot a_2b_2+b_1a_1\ot \{a_2,b_2\}\] makes $P_1\ot P_2$ into a  non-commutative Poisson algebra of	Class 1 satisfying the desired properties.
	
	If $P_1,P_2$ are both non-commutative Poisson algebras of Class 2, then (\ref{53}) holds.  It is easy to verify that the bracket on $P_1\ot P_2$ defined by
	\[\{a_1\ot a_2,b_1\ot b_2\}=[a_1\ot a_2,b_1\ot b_2]\]
	makes $P_1\ot P_2$ into a  non-commutative Poisson algebra of	Class 2 satisfying the desired properties.
	
	Now we provide the definition of degree maps and filtrations on non-commutative Poisson algebras which are suitable for our purpose.  
	\begin{defi}\label{definition2}
		Let $(P,\cdot,\{,\})$ be a non-commutative Poisson algebra. A map $u:P\rightarrow \mathbb Z\cup\{-\infty\}$ is called a degree map if it satisfies that for any $z, w\in P$, 
		
		1. $u(\{z,w\}) \leq u(zw) = u(z) + u(w). $
		
		2. $u(z + w) \leq \max\{u(z), u(w)\}.$
		
		3.    $u(z) = -\infty$ iff $z = 0.$
		
	\end{defi}
	
	\begin{defi}
		A $\mathbb{Z}$-filtration or simply a filtration of the non-commutative Poisson algebra $P$ is a sequence of $K$-subspaces
		$$\cdots\subseteq P_{i-1}\subseteq P_i\subseteq P_{i+1}\subseteq\cdots,\quad i\in \mathbb{Z},$$
		such that 
		$$1 \in P_0,\quad  P_i \cdot P_j \subseteq P_{i + j}\ \text{and}\ \{P_i, P_j\} \subseteq P_{i + j} \text{ for all } i,j\in \mathbb{Z},\quad P = \bigcup_{i\in \mathbb{Z}}P_i. $$
		The filtration $\{P_t\}_{t\in\mathbb Z}$ (or simply  $\{P_t\}$)  will be called finite if each $P_t$ is finite dimensional, and  will be called discrete if $P_t=\{0\}$ for all $t< 0$.		
	\end{defi}
	With the filtration, we can obtain the associated graded non-commutative Poisson algebra.
	\begin{defprop}\label{proposition1}  Let  $\{P_t\}$ be a filtration on the non-commutative Poisson algebra $P$. Let
		$$gr\ P = \bigoplus_{i\in \mathbb{Z}}P_i/P_{i-1},$$
		The product and bracket on $gr\ P$ are defined as follows.	For $ x + P_{i-1} \in P_i/ P_{i-1}, y+P_{j-1}\in P_j/P_{j-1}$,
		let
		$$(x + P_{i-1})\cdot (y + P_{j-1}) = xy + P_{i + j - 1},\quad\{x + P_{i-1}, y + P_{j-1}\} = \{x,y\} + P_{i + j - 1},$$ which are well-defined and can be extended to $gr\ P$ bilinearly.
		Then $gr\ P$	is a graded non-commutative Poisson algebra, which will be called the associated (graded) non-commutative Poisson algebra of $P$.
		\end{defprop}
	The proof is straightforward and is omitted.
	
	If $P$ is of Class 1 (resp. Class 2), then the associated algebra $gr\ P$ is of Class 1 (resp. Class 2).
	
	Given a degree map $u:P\rightarrow \mathbb Z\cup\{-\infty\}$ over the non-commutative Poisson algebra $P$, there exists a  filtration $\{P_t\}$ on $P$ as follows: \[P_t=\{a\in P\mid u(a)\le t\},\quad  t\in\mathbb Z.\] It is clear that $ P_i \cdot P_j \subseteq P_{i + j}\ \text{and}\ \{P_i, P_j\} \subseteq P_{i + j} \text{ for all } i,j\in \mathbb{Z}$,
	which will be referred as the filtration induced by $u$.
	
	Using the definitions above,  it can be concluded that $(P,\ \cdot)$ is a domain at this point. Furthermore, if $(gr\ P,\{,\})$ is commutative, then $\{P_i, P_j\} \subseteq P_{i+j-1}$, so $u(\{z, w\})\leq  u(z) + u(w) - 1$ for all $ z, w \in P.$
	
%
%
%
%
%
%
%
	When the induced filtration $\{P_t\}$ is finite, the non-commutative Poisson algebra $P$ has another property.
	
	\ble
	Assume that  $u:P\rightarrow \mathbb Z\cup\{-\infty\}$ is a degree map on the non-commutative Poisson algebra $P$, $\{P_t\}$ is the induced filtration which is finite. Then for $z\in P$, $u(z) = 0$ iff $z\in K^\times$.
	\ele
	\bp
	The set $B=u^{-1}(0)\cup \{0\}$ is a finite-dimensional domain over $K$, which must be a division algebra. As $K$ is algebraically closed, one has $B=K$. 
	\ep

	Now, we proceed to discuss the classification of the element $z_1\ot z_2$ in $P_1\ot P_2$ (with $z_i\in P_i$ for $i=1,2$) in the following setup: 
	
	Let $P_1$ and $P_2$  be  non-commutative Poisson algebras of the same class (Class 1 or Class 2)
	, and let $u_i:P_i\rightarrow \mathbb{Z}_{\geq 0}\cup \{-\infty\}$ for $i=1,2$ be degree maps on $P_1,P_2$ respectively. Assume that the  filtrations $\{P_{1,t}\}$ and $\{P_{2,t}\}$ induced by $u_1$ and $u_2$ respectively are both finite and discrete. The associated non-commutative Poisson algebras are $gr\ P_1$ and $gr\ P_2$.
	
	The sequences $\{P_{1,t}\ot P_2 \}$ and $\{P_1\ot P_{2,t}\} $ are filtrations on $P_1 \ot P_2$, and we can define  maps $u_1^*, u_2^*$ on $P_1\ot P_2$ as follows:	
	
	For $0\ne T \in P_1\ot P_2$, let $u_1^*(T)$ denote the smallest integer $n$ such that $T \in P_{1,n} \ot P_2$, and let $u_2^*(T)$ denote the smallest integer $m$ such that $T \in P_1 \ot P_{2,m}$. Let $u_1^*(0) = u_2^*(0)=-\infty$. 
	
	\begin{theorem}\label{thm4}
		Take $z_i\in P_i$ for $i=1,2$ and let $ \Theta = z_1 \ot z_2$. Then 
		
		1. Assume $z_i\in Z(P_i)$ for $i=1,2$, then $\Theta \in Z(P_1\ot P_2)$. 
		
		2. Assume exactly one of  $z_i$ belongs to $ Z(P_i)$, $i=1,2$. Without loss of generality, we let $z_1\in Z(P_1), z_2\notin Z(P_2)$. In this case, the possible forms of the four subalgebras related to $\Theta$ are as follows.
		
		$F(\Theta)$: If $z_1\notin K$, then $F(\Theta) = P_1\ot N(z_2)$; if $z_1\in K^\times$, then $F(\Theta) = P_1\ot F(z_2)$. 
		
		$D(\Theta)$: If $z_1\notin K$, then $D(\Theta) = P_1\ot C(z_2)$; if  $z_1\in K^\times$, then $D(\Theta) = P_1\ot D(z_2)$. 
		
		$N(\Theta)$: $N(\Theta) = P_1\ot N(z_2)$.
		
		$C(\Theta)$: $C(\Theta) = P_1\ot C(z_2)$.
		
		\end{theorem}

	\begin{proof}
		1. 	When $z_1\in Z(P_1), z_2\in Z(P_2) $, it is clear that $z_1\ot z_2\in Z(P_1\ot P_2) = \Omega_0(P_1\ot P_2)$.
		
		2. When $z_1\in Z(P_1), z_2\notin Z(P_2)$, let $T = \sum_{i\in I}v_i\ot w_i \in P_1 \ot P_2$. $I$ is a finite set, $\{v_i\}$ are linearly independent over $K$, $\{w_i\}$ are also linearly independent over $K$. Let $ \varphi_2 = ad_{z_2}$. It is easy to obtain by induction that 
		\begin{equation}
			ad_\Theta^n(T) = \sum_{i\in I}v_iz_1^n\ot \varphi_2^n(w_i).
			\label{eq3}\end{equation}
		There is a degree map $u_2$ on $P_2$, so $(P_2,\ \cdot)$ is a domain. Therefore  for any $n\geq 1$, $\{v_iz_1^n\}_i$ are also linearly independent over $K$. Clearly, $P_1 \ot N(z_2) \subseteq N(\Theta)$. If $T \in N(\Theta)$ and $ad_\Theta^n(T) = 0$, then $\varphi_2^n(w_i) = 0$ for all $i \in I$, so $N(\Theta) \subseteq P_1 \ot N(z_2)$, and equality holds. Similarly, $C(\Theta) = P_1 \ot C(z_2)$. The remaining cases to consider are $F(\Theta)$ and $D(\Theta)$.
		
		For $F(\Theta):$
		If $z_1\notin K,$  $u_1(z_1) > 0,$ and there exists $i_0\in I$ such that $w_{i_0}\notin N(P_2)$, then 
		$$u_1^*(ad_{\Theta}^n(T))\geq u_1(v_{i_0}z_1^n) = u_1(v_{i_0})+nu_1(z_1),\quad n\geq 0.$$
		Given $u_1(z_1) > 0$, it follows that $\sup_{n\geq 0}u_1^*(ad_{\Theta}^n(T)) = + \infty,T\notin F(\Theta)$. Conversely, if $T\in F(\Theta)$, then for all $i\in I,w_i\in N(z_2), T\in P_1\ot N(z_2) \subseteq N(\Theta)\subseteq F(\Theta)$, Thus the equalitis hold, and $Ev(\Theta)=\{0\}$.   
		
		
		If $z_1\in K^\times$, $u_1(z_1) = 0$ and there exists $i_0\in I$ such that $w_{i_0}\notin F(z_1)$, then $\sup_{n\geq0}\varphi_2^n(w_{i_0}) = +\infty$. Since  $\{v_iz_1^n\}_i$ are linearly independent over $K$, we have $u_2^*(ad^n_\Theta(T))\geq u_2(\varphi_2^n(w_{i_0}))$, so 
		$$\sup_{n\geq 0}u_2^*(ad_\Theta^n(T)) = +\infty,$$
		which implies $T\notin F(\Theta)$. Conversely, if for all $i\in I,w_i\in F(z_2) $, we can let $W_i = K[\varphi_2](w_i)$, $M = \max_{i\in I}\{u_1(v_i)\}$, then   for all $n\geq 0, u_1(v_iz_1^n) = u_1(v_i)\leq M$, so $v_iz_1^n\in P_{1,M}$, and 
		$$ad^n_\Theta(T)\in \sum_{i\in I}P_{1,M}\ot W_i ,$$
		where $P_{1,M},W_i$ are finite-dimensional over $K$. Thus $T\in F(\Theta)$.
		
		
		For $D(\Theta)$: As we have proved, when $z_1\notin K$ we have $Ev(\Theta)  =\{0\},D(\Theta) = C(\Theta) = P_1\ot C(\Theta)$. When $z_1\in K^\times$, we have
		$$\sum_{i\in I}v_i\ot(z_1\varphi_2(w_i)-\lambda w_i) = 0.$$
		Since $\{v_i\}$ are linearly independent over $K$, $w_i\in D(z_2),T\in P_1\ot D(z_2)\subseteq D(\Theta)$, and the equality holds. 
	\end{proof}
	
	The remaining case to consider is when both $z_1$ and $z_2$ are non-central elements. We have the following theorem. 
	\bthm	\label{thm2}
	Assume $z_1\notin Z(P_1),z_2\notin Z(P_2)$ and  the associated non-commutative Poisson algebras $gr\ P_1,gr\ P_2$ are both  commutative under Poisson brackets. Let $\Theta = z_1\ot z_2,$ then
	\[F(\Theta) = N(\Theta) = N(z_1)\ot N(z_2  ),\quad  D(\Theta) =  C(\Theta),\quad Ev(\Theta)=\{0\}.\]			
	\ethm
	\begin{proof}
		Take $v \in P_1, w \in P_2$, let 
		$$\varphi_1 = ad_{z_1},\quad \varphi_2 = ad_{z_2},\quad R_1(v) = vz_1,\quad L_2(w) = z_2w.$$
		The equations $\varphi_1R_1 = R_1\varphi_1,\varphi_2L_2 = L_2\varphi_2$ hold, and 
		$$(\varphi_1\otimes L_2)\cdot(R_1\ot\varphi_2)=(R_1\ot\varphi_2)\cdot(\varphi_1\ot L_2).$$
		Regardless of whether $P$ is of Class 1 or Class 2, one has 
		\begin{align}
			ad_\Theta(v\ot w)
			& = \{z_1,v\}\ot z_2w + vz_1 \ot \{z_2,w\}\notag\\
			& = \varphi_1(v)\ot L_2(w) + R_1(v) \ot \varphi_2(w)\notag\\
			& = (\varphi_1 \ot L_2 + R_1\ot \varphi_2)(v \ot w).\notag
		\end{align}
		Thus 
		\begin{align}
			ad_\Theta^n (v \ot w) &= (\varphi_1 \ot L_2 + R_1 \ot \varphi_2)^n(v \ot w)\notag\\
			& = \sum_{k=0}^{n}\binom{n}{k}\varphi_1^k(v)z_1^{n-k}\ot z_2^k\varphi_2^{n-k}(w)\notag\\
			& = vz_1^n \ot \varphi_2^n(w) + \sum_{k=1}^{n}\binom{n}{k}\varphi_1^k(v)z_1^{n-k}\ot z_2^k\varphi_2^{n-k}(w) .\label{1}
		\end{align}
		If $\varphi_1^m(v) = 0, \varphi_2^l(w) = 0,$ where $ m,l>0$, then $ad_\Theta^{m+l}(v\ot w) = 0$. Therefore 
		\begin{align}
			N(z_1)\ot N(z_2) \subseteq N(\Theta)\subseteq F(\Theta).
			\label{neq4}\end{align}
		
		Define the natural map $\pi:P_2\rt P_2/N(z_2), w\mapsto \bar{w} $. If $T\notin P_1\ot N(z_2)$, then let 
		$$T = \sum_{i\in I}v_i\ot w_i + \sum_{j\in J}\widetilde{v}_j\ot \widetilde{w}_j,$$
		where $I$ and $J$ are finite sets, $I\ne \emptyset$, and $\{\bar{w_i}\} $ are linearly independent over  $K$. If $J\ne \emptyset$, then $\{\widetilde{w}_j\}\subseteq  N(z_2)$ are also linearly independent over $K$. When the coefficients $c_i\in K$ are not all zero, $\sum_{i\in I}c_iw_i\notin N(z_2)$. Furthermore, for $ n\geq 1$,
		$$\sum_{i\in I}c_i\varphi_2^n(w_i) = \varphi_2^n(\sum_{i\in I}c_iw_i)\neq 0.$$
		Hence, $\{\varphi_2^n(w_i)\}_{i \in I}$ are linearly independent over $K$.
		
		It is known that $(gr\ P_1,\{,\})$ is commutative, thus 
		$$u_1(\varphi_1(v)) \leq u_1(v) + u_1(z_1) - 1,\quad v \in P_1.$$
		Define $M_n = \max_{i\in I}\{u_1(v_i)\}+nu_1(z_1)$. Using the above inequality, for $1 \leq k \leq n$, 
		\begin{align}
			u_1(\varphi_1^k(v_i)z_1^{n-k}) & \leq u_1(v_i) + k(u_1(z_1) - 1) + (n - k)u_1(z_1)\notag\\
			& = u_1(v_i) + nu_1(z_1) - k \notag\\
			& \leq M_n - k. \notag
		\end{align}
		For any $j \in J$, there exists $k_j \geq 0 $ such that $\varphi_2^{k_j}(\widetilde{w}_j) \neq 0, \varphi_2^{k_j+1}(\widetilde{w}_j) = 0$. When $n,k$ satisfy $0\le n-k\le k_j $, the following holds:
		\begin{align}
			u_1(\varphi_1^k(\widetilde{v}_j)z_1^{n-k}) & \leq u_1(\widetilde{v_j}) + nu_1(z_1) - k\notag\\
			& = u_1(\widetilde{v}_j) - \max_{i \in I}\{u_1(v_i)\} + \max_{i \in I}\{u_1(v_i)\} + nu_1(z_1)- n +  n - k\notag\\
			& \leq  u_1(\widetilde{v}_j) - \max_{i \in I}\{u_1(v_i)\} + M_n - n + k_j \notag\\
			& = M_n - n  + l_j,\quad l_j = u_1(\widetilde{v}_j) - \max_{i \in I}\{u_1(v_i)\} + k_j. \notag 
		\end{align}
		For any $ j \in J,$ $l_j$ is a constant. Therefore when $n > \max_{j\in J}\{k_j, l_j\}$, the following two inequalities hold:
		\begin{equation}
			u_1(\varphi_1^k(v_i)z_1^{n-k})<M_{n},\quad 1\leq k\leq n,i\in I,
			\notag
		\end{equation}
		\begin{equation}
			u_1(\varphi_1^k(\widetilde{v}_j)z_1^{n-k})<M_{n},\quad n-k_j\leq k\leq n,j\in J.
			\notag
		\end{equation}
		From equation (\ref{1}), it can be concluded that:
		\begin{align*}
			ad_\Theta^{n}(T)&=\sum_{i\in I}(v_iz_1^{n} \ot \varphi_2^{n}(w_i) +\sum_{k=1}^{n}\binom{n}{k}\varphi_1^k(v_i)z_1^{n-k}\ot z_2^k\varphi_2^{n-k}(w_i)) + \\ &\ \ \ \sum_{j\in J}\sum_{k:0\le n-k\le k_j}\binom{n}{k}\varphi_1^k(\tilde v_j)z_1^{n-k}\ot z_2^k\varphi_2^{n-k}(\tilde w_j) \\
			& = \sum_{i:u_1(v_i)=M_0}v_iz_1^{n}\ot \varphi_2^{n}(w_i)+\sum_{h}\widehat{v}_h\ot \widehat{w}_h,\quad u_1(\widehat{v}_h)<M_{n}.
		\end{align*}
		$\{\varphi_2^{n}(w_i)\}_{i\in I}$ are linearly independent over $K$, so 
		$$u_1^*(ad_\Theta^{n}(T)) = M_{n}, \quad n > \max_{j\in J}\{k_j, l_j\}.$$
		Also, because $z_1\notin Z(P_1), u_1(z_1)>0$, we have $M_n < M_{n+1}$, so $T \notin F(\Theta)$. Conversely $T\in F(\Theta)$ implies $T\in P_1\ot N(z_2),F(\Theta)\subseteq P_1\ot N(z_2)$. 
		By symmetry, we have $T \in N(z_1) \ot P_2$, so $F(\Theta)\subseteq N(z_1)\ot N(z_2)$.
		
		From equation (\ref{neq4}) and Lemma \ref{lemma0}, we can conclude that $F(\Theta) = N(\Theta) = N(z_1) \ot N(z_2),$ $D(\Theta) = C(\Theta)$, and $Ev(\Theta) =\{0\}$.
	\end{proof}	
	
	Based on the above theorem, we can further discuss the possible types of $\Theta$ when both $z_1$ and $z_2$ are non-central elements. 
	\begin{coro}\label{coro4.1}
		With the same conditions  as in Theorem \ref{thm2}. 
		
		If  $C(z_1)\subsetneq N(z_1)$ or $C(z_2)\subsetneq N(z_2)$, then $\Theta$ is nilpotent, and $\Theta$ is strictly nipotent iff $z_1,z_2$ are strictly nilpotent. 
		
		If $N(z_1) = C(z_1)$ and $N(z_2) = C(z_2)$, then $\Theta \in \Omega_0'(P_1\ot P_2)$. 
		
		If $N(z_1) = C(z_1)$ or $N(z_2) = C(z_2)$, then $C(\Theta) = C(z_1) \ot C(z_2)$.
		\end{coro}
	\begin{proof}
		If $C(z_1)\subsetneq N(z_1)$, take $v \in N(z_1)\setminus C(z_1)$. Then $v\ot 1 \in N(\Theta) $ while
		$$ad_\Theta (v\ot 1) = \{z_1,v\}\ot z_2 \neq 0.$$
		Thus $C(\Theta)\subsetneq N(\Theta)$, and $\Theta$ is nilpotent by Theorem \ref{thm2}. Similarly, when $C(z_2)\subsetneq N(z_2)$, we can deduce that $\Theta$ is nilpotent. Theorem \ref{thm2} also tells us that $F(\Theta) = P_1\ot P_2$ iff $N(z_1) = P_1,N(z_2) = P_2$, which means that $z_1, z_2$ are strictly nilpotent.
		
		If $N(z_1) = C(z_1)$ and $N(z_2) = C(z_2)$, then 
		$$ F(\Theta) = C(z_1) \ot C(z_2) \subseteq C(\Theta) \subseteq F(\Theta),$$
		with equalities holding. Additionally, since $z_1 \notin Z(P_1), C(z_1) \subsetneq P_1$, it follows that $\Theta\in \Omega_0'(P_1\ot P_2)$.
		
		If $N(z_1) = C(z_1)$, once again, let $T = \sum_{i\in I}v_i\ot w_i \in C(\Theta),$ where $I$ is a finite set and $\{v_i \}\subseteq C(z_1)$ are linearly independent over $K$. Since there is a degree map $u_1$ on $P_1$, $(P_1,\ \cdot)$ is a domain, and $\{v_iz_1\}$ are linearly independent over $K$. It's known that 
		$$ad_\Theta(T) = \sum_{i\in I} (\varphi_1(v_i) \ot L_2(w_i) + R_1(v_i) \ot \varphi_2(w_i)) = \sum_{i \in I} v_iz_1 \ot \varphi_2(w_i) = 0.$$
		Therefore $ \varphi_2(w_i) = 0, w_i \in C(z_2), \forall i\in I$. Thus $T\in C(z_1) \ot C(z_2) \subseteq C(\Theta)$,  and the equality holds. The same reasoning applies when $N(z_2) = C(z_2)$, leading to $C(\Theta) = C(z_1)\ot C(z_2)$. 
	\end{proof}
	
	Next we apply the above results to the Weyl algebra.
	
	For a general non-commutative Poisson algebra $P$, given $z \in P$, the relationship between $F(z)$ and $N(z),D(z)$ could be a proper inclusion rather than equality. However, the situation is relatively simple in $A_1$, because for any $z\in A_1,$ at least one of the equalities $F(z) = N(z)$ or $F(z) = D(z)$  holds{\cite{d}}.
	
	Let $ p, q $ be the generators of $ A_1 $, and $ [q, p] = 1 $. Define a degree map on $ A_1 $: 
	$$ \beta_{1}: A_1 \rightarrow \mathbb{Z}_{\geq 0} \cup \{-\infty\},\quad 0 \mapsto -\infty, \quad 0\ne \sum_{i,j \in \mathbb{Z}_{\geq 0}} a_{ij} p^i q^j \mapsto \max\{i + j \mid a_{ij} \neq 0\}. $$
	Let $ P_1 = A_1 $, $ u_1 = \beta_1 $. In this case, the induced filtration $\{P_{1,t}\}$ is finite and discrete, and $ gr\ P_1 = K[X,Y] $, so $ A_1 $ satisfies the above theorem and corollary. For example, taking $ z \in A_1 $, by Theorem \ref{thm4}, the type of $ z \otimes 1 $ in $ A_2 $ is the same as the type of $ z $ in $ A_1 $.
	
	For the $ n $-th Weyl algebra $ A_n $ with $ n \geq 2 $, the Berstein filtration on it is  finite and discrete with $gr\ A_n$ a commutative domain.
	Therefore, it also satisfies the conditions of the above theorem.
	
	We have the following corollary:
	\begin{coro}	\label{coro5}
		Let $ z_1, z_2 \in A_1 \setminus  K $. If $ z_1 $ or $ z_2 $ is nilpotent, then $ z_1 \otimes z_2 $ is nilpotent in $ A_2 $, and $ z_1 \otimes z_2 $ is strictly nilpotent iff both $ z_1 $ and $ z_2 $ are strictly nilpotent.
		
		If neither $ z_1 $ nor $ z_2 $ is nilpotent, then $ z_1 \otimes z_2 \in \Omega_0'(A_2) $.
		
		If $ z_1 $ or $ z_2 $ is not nilpotent, then $ C(z_1 \otimes z_2) = C(z_1) \otimes C(z_2) $, which is commutative. If both $ z_1 $ and $ z_2 $ are semisimple, then $$ C(z_1 \otimes z_2) = K[z_1] \otimes K[z_2]. $$
	\end{coro}
	\begin{proof}
		It is known that the centralizer of non-constant elements in $ A_1 $ commutes{\cite{a}}. Furthermore, Joseph proved that for $ x \in \Omega_2(A_1) \cup \Omega_2'(A_1) $, $ C(x) = K[x] ${\cite{J1975}}, and by applying Theorem \ref{thm2} and Corollary \ref{coro4.1}, the result follows.
	\end{proof}
	Bavula pointed out in {\cite{VB2005}} that if the type of $ z \in A_1 \setminus K $ is known, and $ f(X) \in K[X] $ with $ \deg f > 1 $, then the possible types of $ f(z) $ are as follows:
	
	1. If $ z $ is not semisimple, then the type of $ f(z) $ is the same as that of $ z $.
	
	2. If $ z $ is semisimple, then $ f(z) \in \Omega_0'(A_1) $.
	
	\noindent By Corollary \ref{coro5}, we can draw the following conclusion:
	
	\begin{coro}\label{coro7}
		Given \( f, g \in K[X] \), \( \deg f, \deg g > 1 \), and \( z_1, z_2 \in A_1 \setminus K \), then \( f(z_1) \otimes g(z_2) \) and \( z_1 \otimes z_2 \) have the same type in \( A_2 \).			
	\end{coro}
	\begin{proof}
		If $ z_1 $ or $ z_2 $ is nilpotent, then $ f(z_1) $ or $ g(z_2) $ is nilpotent. By Corollary \ref{coro5}, it follows that $ z_1 \otimes z_2 $ and $ f(z_1) \otimes g(z_2) $ are both nilpotent. $ f(z_1) \otimes g(z_2) $ is strictly nilpotent iff $ f(z_1), g(z_2) \in \Omega_1(A_1) $, iff $ z_1, z_2 \in \Omega_1(A_1) $, iff $ z_1 \otimes z_2 $ is strictly nilpotent. Therefore, in this case, the type of $ z_1 \otimes z_2 $ is the same as that of $ f(z_1) \otimes g(z_2) $.
		
		If neither $ z_1 $ nor $ z_2 $ is nilpotent, then $ f(z_1), g(z_2) \in \Omega_0'(A_1) $. Hence, by Corollary \ref{coro5}, $ z_1 \otimes z_2, f(z_1) \otimes g(z_2) \in \Omega_0'(P_1 \otimes P_2) $, and they have the same type.
	\end{proof}
	By analogy with Proposition \ref{lemma1}, we have:
	\begin{coro}
		Let $ \varphi: A_1 \rightarrow A_1 $ and $ \psi:A_1 \rightarrow A_1 $ be endomorphisms such that Im$\,\varphi \subsetneq A_1$,Im$\,\psi \subsetneq A_1$. Let $ \xi = \varphi \otimes \psi: A_2 \rightarrow A_2, v \otimes w \mapsto \varphi(v) \otimes \psi(w) $. If $ z_1, z_2 \in A_1 \setminus K $ and $ \Theta = z_1 \otimes z_2 $,  then 
		$$ F(\xi(\Theta)) = N(\varphi(z_1)) \otimes N(\psi(z_2)). $$
		
		If there is a nilpotent element among $ z_1 $ or $ z_2 $, then $ \xi(\Theta) \in \Omega_1'(A_2) $;
		If $ z_1 $ and $ z_2 $ are both semisimple elements, then $ \xi(\Theta) \in \Omega_0'(A_2) $.
	\end{coro} 
	\begin{proof}
		If there is a nilpotent element among $ z_1 $ or $ z_2 $, let's assume it is $ z_1 $. By Proposition \ref{lemma1}, we know that $ \varphi(z_1) \in \Omega_1'(A_1), C(\varphi(z_1)) \subsetneq N(\varphi(z_1)) \subsetneq A_1 $. Therefore, we have:
		$$ C(\xi(\Theta)) \subsetneq N(\varphi(z_1)) \otimes N(\psi(z_2)) = F(\xi(\Theta)) \subsetneq A_2, $$
		and thus $ \xi(\Theta) \in \Omega_1'(A_2) $.
		
		If $ z_1 $ and $ z_2 $ are both semisimple elements, by Proposition \ref{lemma1}, we have $ \varphi(z_1), \psi(z_2) \in \Omega_2'(A_1) $, and $ C(\varphi(z_1)) = N(\varphi(z_1)) \subsetneq A_1, C(\psi(z_2)) = N(\psi(z_2)) \subsetneq A_1 $. Therefore, 
		$$ F(\xi(\Theta)) = C(\varphi(z_1)) \ot C(\psi(z_2)) \subseteq C(\xi(\Theta)) \subseteq F(\xi(\Theta)), $$ 
		the equalities hold, and $ \xi(\Theta) \in \Omega_0'(A_2) $.
	\end{proof}

	\section{The classification of elements of the form $  z_1\ot 1  + 1\ot z_2 $}
	\setcounter{equation}{0}\setcounter{theorem}{0}
	
	Now we will discuss the classification of the element $z_1\ot 1+1\ot z_2$ in $P_1\ot P_2$ (with $z_i\in P_i$ for $i=1,2$) in the same setup as in last section:
	
	Let $P_1$ and $P_2$  be  non-commutative Poisson algebras of the same class (Class 1 or Class 2), and let $u_i:P_i\rightarrow \mathbb{Z}_{\geq 0}\cup \{-\infty\}$ for $i=1,2$ be degree maps on $P_1,P_2$ respectively. Assume that the  filtrations $\{P_{1,t}\}$ and $\{P_{2,t}\}$ induced by $u_1$ and $u_2$ respectively are both finite and discrete.
	
	%

	\begin{theorem}	\label{thm1}
		Take $ z_1 \in P_1, z_2 \in P_2, \Gamma = z_1 \otimes 1 + 1 \otimes z_2 \in P_1 \otimes P_2 $, then
		$$ F(\Gamma) = F(z_1) \otimes F(z_2). $$
	\end{theorem}
	
	We make a remark before the proof. Unlike Theorem \ref{thm2}, Theorem \ref{thm1} does not require $gr\ P_1$ and $gr\ P_2$ to be commutative under brackets. All other results in this section are related to Theorem \ref{thm1}, and similarly do not require $gr\ P_1$ and $gr\ P_2$  to be commutative under brackets.
	
	\begin{proof}
		If $z_1\in Z(P_1),z_2\in Z(P_2)$, it is evident that $\Gamma \in Z(P_1\ot P_2), F(\Gamma) = P_1\ot P_2 = F(z_1)\ot F(z_2)$. 
		
		If $z_1\in Z(P_1), z_2\in P_2\setminus Z(P_2)$, then by Theorem \ref{thm4},  $F(\Gamma) = F(1\ot z_2) = P_1\ot F(z_2) = F(z_1)\ot F(z_2)$. 
		
		If $z_1\in P_1\setminus Z(P_1), z_2\in Z(P_2)$, by symmetry, it follows that $ F(\Gamma) = F(z_1) \ot F(z_2)$. 
		
		If $z_1\in P_1\setminus Z(P_1), z_2\in P_2\setminus Z(P_2)$, suppose $0\ne T = \sum_{i\in I} v_i\otimes w_i\in P_1\otimes P_2$, where $I$ is a finite set, $\{v_i\}$ are linearly independent over $K$, and $\{w_i\}$  are also  linearly independent over $K$.
		
		Let $\varphi_1 = ad_{z_1}, \varphi_2 = ad_{z_2}, V_i = K[\varphi_1](v_i),W_i = K[\varphi_2](w_i)$, then we have
		\begin{align}
			ad_\Gamma^n(T) &= (\varphi_1  \otimes 1+1 \otimes \varphi_2)^n(\sum_{i\in I}v_i\otimes w_i)\notag \\
			&=\sum_{i\in I}\sum_{t=0}^{n}\binom{n}{t}\varphi_1^t(v_i)\otimes\varphi_2^{n-t}(w_i)\in \sum_{i\in I}V_i\otimes W_i.\label{eq1}
		\end{align}
		If for all $i \in I,$ $v_i\in F(z_1), w_i\in F(z_2)$, then $\dim_K (V_i) < +\infty, \dim_K(W_i) < +\infty$, and 
		$$\dim_K(\sum_{i\in I}V_i \otimes W_i) < +\infty.$$
		Thus $T \in F(\Gamma)$, and $F(z_1) \otimes F(z_2) \subseteq F(\Gamma)$. 
		
		If $ T\notin F(z_1) \otimes F(z_2)$, then there exists $i \in I$ such that 
		$v_i \notin F(z_1) $ or $w_i \notin F(z_2) $.
		
		Assume	$w_j \notin F(z_2) ,j \in I$. Since $\{P_{2,t}\}$ is finite, we have  $\sup_{n\geq 0}u_2(\varphi_2^n(w_j))=+\infty$. Let 
		$$M_k = \max_{0\leq n \leq k,i \in I}\left\{ u_2(\varphi^n_2(w_i))\right\}.  $$
		In particular $M_0 = \max_{i\in I} \{u_2(w_i)\}$, and 
		$$\sup_{k\geq 0}{M_k}= +\infty.$$
		Let $k_1$ be the smallest positive integer such that $M_{k} > M_0$ holds, then 
		$$u_2(\varphi_2^n(w_i))< M_{k_1},\quad 0\leq n < k_1 , i \in I.$$
		Thus 
		$$ad_\Gamma^{k_1}(T) = \sum_{i:u_2(\varphi_2^{k_1}(w_i))=M_{k_1}}v_i \otimes \varphi_2^{k_1}(w_i) + \sum_h \widehat{v}_h\otimes \widehat{w}_h,\quad u_2(\widehat{w}_h)<M_{k_1}.$$
		From the linear independence of  $\{v_i\}$ over $K$, it follows that $u_2^*(ad_\Gamma^{k_1}(T))= M_{k_1}$.
		
		By repeating the above steps, we obtain a sequence of positive integers: $0 < k_1 < k_2 <\cdots <k_m < \cdots$, such that the following holds:
		$$u_2^*(ad_\Gamma^{k_m}(T))=M_{k_m},\quad M_{k_m}<M_{k_{m+1}},m\geq 1.$$
		Thus $\dim_K\ K[ad_\Gamma](T)=+\infty ,T \notin F(\Gamma)$. Conversely, if $T
		\in F(\Gamma)$, then $T \in P_1 \ot F(z_2)$.
		
		By symmetry, we can conclude from $T \in F(\Gamma)$ that $T \in F(z_1) \ot P_2$. Therefore $F(\Gamma)\subseteq F(z_1)\ot F(z_2)$, and equality holds.
	\end{proof}
	\begin{rem}\label{rem0}
		Using equation (\ref{eq1}), we have the following obvious inclusions:
		$$D(z_1)\ot D(z_2) \subseteq D(\Gamma),$$
		$$N(z_1) \ot N(z_2) \subseteq N(\Gamma),$$
		$$C(z_1) \ot C(z_2) \subseteq \bigoplus_{\lambda\in K} D(z_1,\lambda) \ot D(z_2,-\lambda)\subseteq C(\Gamma).$$
		\end{rem}
	Given that $F(\Gamma) = \bigoplus_{\lambda \in K} F(\Gamma, \lambda),F(\Gamma, \lambda) = \bigcup_{n\geq 1}\ker (ad_\Gamma - \lambda)^n$, using equation (\ref{eq1}), we can obtain the following characterization of $F(\Gamma)$ and $N(\Gamma)$.
	\begin{lemma}\label{lemma2}
		Take  $\lambda\in K$. We have
		$$F(\Gamma,\lambda) = \bigoplus_{\mu \in K}F(z_1,\mu)\ot F(z_2,\lambda - \mu).$$
		In particular, when $\lambda = 0$ we have 
		$$ N(\Gamma) = F(\Gamma, 0) = \bigoplus_{\lambda \in K} F(z_1, \lambda) \ot F(z_2, -\lambda). $$
		Let $ -Ev(z_2) = \{\lambda \mid -\lambda \in Ev(z_2)\} $, then $ N(\Gamma) = N(z_1) \ot N(z_2) $ iff
		\begin{equation}
			Ev(z_1)\cap (-Ev(z_2)) = \{0\}.\label{eq2}
		\end{equation}
		\end{lemma}
	\begin{proof}
		Let $ v \in F(z_1, a) , w \in F(z_2, b) $, where $ a, b \in K $. Then there exist positive integers $ n_1 , n_2 $ such that 
		$$ (\varphi_1 - a)^{n_1}(v) = 0, \quad (\varphi_2 - b)^{n_2}(w) = 0, $$ where $ \varphi_1 = ad_{z_1} $ and $ \varphi_2 = ad_{z_2} $. It is known that 
		$$ ad_\Gamma - (a + b) = (\varphi_1 - a) \ot 1 + 1 \ot (\varphi_2 - b).$$ 
		Hence
		\begin{align}
			[ad_\Gamma - (a+b)]^{n_1 +n_2}(v\ot w) 
			& =\sum_{t = 0}^{n_1 +n_2}\binom{n_1 +n_2}{t}(\varphi_1-a)^t(v)\ot (\varphi_2-b)^{n_1 + n_2 -t}(w)\notag\\
			& = 0.\notag
		\end{align}
		Thus $v\ot w\in F(\Gamma,a + b),$ and we have 
		$$\bigoplus_{\mu \in K}F(z_1,\mu)\ot F(z_2,\lambda - \mu)\subseteq F(\lambda, \Gamma).$$
		
		Let $ T = \sum_{i \in I} v_i \ot w_i \in F(\Gamma, \lambda) $, where $ I $ is a finite set. For any $ i \in I $, by Theorem \ref{thm1}, we can assume that $ v_i \in F(z_1, a_i) , w_i \in F(z_2, b_i) , a_i, b_i \in K $, and that $ \{ v_i \ot w_i \} $ is linearly independent over $ K $. Since $ v_i \ot w_i \in F(\Gamma, a_i + b_i) $, we have $ a_i + b_i = \lambda $, and the equation holds.
		
		Finally, $ N(\Gamma) = N(z_1) \ot N(z_2) $ if and only if
		$$\bigoplus_{\lambda \in K^\times} F(z_1, \lambda) \ot F(z_2, -\lambda) = 0, $$ if and only if
		$ Ev(z_1) \cap (-Ev(z_2)) = \{ 0 \}. $
	\end{proof}
		%
	In special cases, the equation (\ref{eq2}) can hold, such as when $ Ev(z_1) = \{ 0 \} $. So we have the following corollary.
	\begin{coro}\label{cor1}
		When $ F(z_1) = N(z_1) $ or $ F(z_2) = N(z_2) $, one has
		$$N(\Gamma) = N(z_1)\ot N(z_2).$$
	\end{coro}
	\begin{proof}
		We may assume $ F(z_1) = N(z_1) $. We have $ F(z_1) = N(z_1) = F(z_1, 0) $ iff $ Ev(z_1) = \{ 0 \} $, and hence by Lemma \ref{lemma2}, we conclude that $ N(\Gamma) = N(z_1) \ot N(z_2). $
	\end{proof}
	We have characterized  $ N(\Gamma) $ in $ P_1 \ot P_2 $, along with a condition for equality in Remark \ref{rem0}. The subalgebra $ D(\Gamma) $ is not easy to characterize. We present the following conclusion, which is similar to Corollary \ref{cor1}: 
	\begin{coro}\label{cor2}
		When $ F(z_1) = D(z_1) $ or $ F(z_2) = D(z_2) $, for any $\lambda \in K $ one has
		$$ D(\Gamma, \lambda) = \bigoplus_{\mu \in K} D(z_1, \mu) \ot D(z_2, \lambda - \mu), $$
		and in this case $ D(\Gamma) = D(z_1) \ot D(z_2) $.
		\end{coro}
	\begin{proof}
		Let $ T = \sum_{i \in I} v_i \ot w_i \in D(\Gamma, \lambda) \subseteq F(\Gamma, \lambda) , \lambda \in K , I $ is a finite set. From Lemma \ref{lemma2}, we know that for all $i \in I $, we can assume $ v_i \in F(z_1, a_i) , w_i \in F(z_2, b_i) $, with $ \{ v_i \} $ being linearly independent over $ K $, and $ a_i + b_i = \lambda $. We assume $ F(z_1) = D(z_1) $, and let $ \varphi_1 = ad_{z_1} $, $ \varphi_2 = ad_{z_2} $, then 
		\begin{align}
			(ad_\Gamma-\lambda)(T) &= \sum_{i\in I} (\varphi_1(v_i)\ot w_i + v_i\ot \varphi_2(w_i) - \lambda v_i\ot w_i)\notag\\
			&= \sum_{i\in I} (a_iv_i\ot w_i + v_i\ot \varphi_2(w_i) - v_i\ot\lambda  w_i)\notag\\
			&= \sum_{i\in I} v_i\ot (\varphi_2(w_i)-(\lambda - a_i)w_i)\notag\\
			& = 0.\notag
		\end{align}
		Since $\{v_i\}$ are linearly independent over $K$, for all $i\in I, w_i\in D(z_2, \lambda - a_i)$. Therefore $$T\in\bigoplus_{\mu \in K}D(z_1,\mu)\ot D(z_2,\lambda - \mu)\subseteq D(\Gamma,\lambda),$$
		with equality holding. $D(\Gamma) =\bigoplus_{\lambda\in K} D(\Gamma,\lambda) = D(z_1)\ot D(z_2)$. The same conclusion holds when $F(z_2)= D(z_2)$.
	\end{proof}
	Now we can attempt to determine the type of $\Gamma$ in $P_1\ot P_2$. If $z_1 \in  Z(P_1)$ or $z_2\in Z(P_2)$, the type of $\Gamma$ is the same as that of $1\ot z_2$, and the type of $1\ot z_2$ can be determined using Theorem \ref{thm4}. In the case where neither $z_1$ nor $z_2$ is central, we will determine the type of $\Gamma$ under various conditions. 
	\begin{coro}\label{coro1}
		If $F(z_1) = N(z_1)$ and $ F(z_2) = N(z_2)$, then 
		$$F(\Gamma) = N(\Gamma) = N(z_1)\ot N(z_2),\quad  D(\Gamma) = C(\Gamma),\quad Ev(\Gamma)=\{0\}.$$
		Moreover, if $C(z_1)\subsetneq N(z_1)$ or $C(z_2)\subsetneq N(z_2)$, then 
		$$D(\Gamma)=C(\Gamma)\subsetneq N(\Gamma)=F(\Gamma).$$
		In this case, $\Gamma \in \Omega_1 \cup \Omega_1'$ is a nilpotent element. 
		\end{coro}
	\begin{proof}
		From Theorem \ref{thm1}, it follows that $F(\Gamma) = N(z_1) \otimes N(z_2)\subseteq N(\Gamma)\subseteq F(\Gamma)$, with equalities holding. From Lemma \ref{lemma0}, we know that $D(\Gamma) = C(\Gamma), Ev(\Gamma) = \{0\}$.
		
		If $C(z_1)\subsetneq N(z_1)$, we choose $v\in N(z_1)\setminus C(z_1)$. Then $v \otimes 1\in N(\Gamma) \setminus C(\Gamma) $, so $C(\Gamma)\subsetneq N(\Gamma)$. Similarly, if $C(z_2)\subsetneq N(z_2)$, we can prove that  $C(\Gamma)\subsetneq N(\Gamma)$.
	\end{proof}
	Similarly, it can be proved that:
	\begin{coro} \label{coro2}
		If $F(z_1) = D(z_1), F(z_2) = D(z_2)$, then 
		$$F(\Gamma) = D(\Gamma) = D(z_1)\ot D(z_2),\quad  N(\Gamma) = C(\Gamma).$$
		Furthermore, if $C(z_1)\subsetneq D(z_1)$ or $C(z_2)\subsetneq D(z_2)$, then 
		$$N(\Gamma) = C(\Gamma)\subsetneq D(\Gamma)=F(\Gamma),$$
		and in this case $\Gamma \in \Omega_2\cup\Omega_2'$ is semisimple.
		\end{coro}
	\begin{coro} \label{coro3}
		If $F(z_1) = N(z_1), F(z_2) = D(z_2)$, then 
		$$F(\Gamma) = N(z_1) \otimes D(z_2),$$
		$$N(\Gamma) = N(z_1) \otimes N(z_2) = N(z_1) \otimes C(z_2),$$
		$$D(\Gamma) = D(z_1) \otimes D(z_2) = C(z_1) \otimes D(z_2),$$
		$$C(\Gamma) = C(z_1) \otimes C(z_2).$$
		If $C(z_1)\subsetneq N(z_1)$, then $C(\Gamma)\subsetneq N(\Gamma)$; if $C(z_2)\subsetneq D(z_2)$, then  $C(\Gamma)\subsetneq D(\Gamma)$; if $C(z_1)\subsetneq N(z_1) $ and $ C(z_2)\subsetneq D(z_2)$, then 
		$$N(\Gamma)\supsetneq C(\Gamma) \subsetneq D(\Gamma)\subsetneq F(\Gamma),$$
		and in this case, $\Gamma \in \Omega_3\cup\Omega_3'$ is of Jordan type.
		\end{coro}
	\begin{proof}
		From Lemma \ref{lemma0}, it's known that, $D(z_1) = C(z_1), N(z_2) =C(z_2)$. From Theorem \ref{thm1}, it's known that 
		$$F(\Gamma) = F(z_1)\ot F(z_2)= N(z_1) \otimes D(z_2).$$
		From Corollary \ref{cor1}, it's known that 
		$N(\Gamma) = N(z_1)\ot N(z_2) = N(z_1)\ot C(z_2)$.
		From Corollary \ref{cor2}, it's known that 
		$D(\Gamma) = D(z_1)\ot D(z_2) = C(z_1)\ot D(z_2)$.
		Therefore, 
		$$C(\Gamma) = (N(z_1)\ot C(z_2)) \cap (C(z_1)\ot D(z_2)) = C(z_1)\ot C(z_2).$$
		\noindent If $C(z_1)\subsetneq N(z_1)$, then 
		$$C(\Gamma) = C(z_1)\ot C(z_2)\subsetneq N(z_1)\ot C(z_2)  =N(\Gamma).$$
		If $C(z_2) \subsetneq D(z_2) $, then 
		$$C(\Gamma) = C(z_1)\ot C(z_2)\subsetneq C(z_1)\ot D(z_2)  =D(\Gamma).$$
		
		\noindent If $C(z_1)\subsetneq N(z_1)$ and $C(z_2) \subsetneq D(z_2) $, from Lemma \ref{lemma3}, it's known that $\Gamma\in \Omega_3\cup\Omega_3'$.
	\end{proof}
	The following conclusion can be proved in a similar way. 
	\begin{coro}
		If $F(z_1) = C(z_1)\subsetneq P_1$ i.e. $z_1\in \Omega_0'(P_1)$, then 
		$$F(\Gamma) = C(z_1) \otimes F(z_2),\quad N(\Gamma) = C(z_1) \otimes N(z_2),$$
		$$D(\Gamma) = C(z_1) \otimes D(z_2),\quad C(\Gamma) = C(z_1) \otimes C(z_2).$$
		In this case, the possible types of $\Gamma$ are as follows:
		
		\noindent If $z_2\in \Omega_0(P_1)\cup \Omega_0'(P_1)$, then $\Gamma\in \Omega_0'(P_1\ot P_2)$;
		
		\noindent If $z_2\in \Omega_1(P_1)\cup \Omega_1'(P_1)$, then $\Gamma\in \Omega_1'(P_1\ot P_2)$;
		
		\noindent If $z_2\in \Omega_2(P_1)\cup \Omega_2'(P_1)$, then $\Gamma\in \Omega_2'(P_1\ot P_2)$;
		
		\noindent If $z_2\in \Omega_3(P_1)\cup \Omega_3'(P_1)$, then $\Gamma\in \Omega_3'(P_1\ot P_2)$.
	\end{coro}
	\begin{coro}
		Let $z_1\in \Omega_3(P_1)\cup \Omega_3'(P_1),z_2\in P_2$, then $\Gamma\in \Omega_3(P_1\ot P_2)\cup\Omega_3'(P_1\ot P_2)$.
	\end{coro}
	\begin{proof}
		It is known from Lemma \ref{lemma2} that 
		$$Ev(z_1) + Ev(z_2) = Ev(\Gamma).$$
		Since $z_1\in \Omega_3(P_1)\cup\Omega_3'(P_1)$, we have $\{0\} \subsetneq Ev(z_1),\{0\}\subsetneq Ev(\Gamma),C(\Gamma)\subsetneq D(\Gamma)$. 
		
		Let $\varphi_1 = ad_{z_1}$. There exists $\lambda\in Ev(z_1),w\in F(z_1, \lambda)$ and $n\geq 1$ such that 
		$$(\varphi_1 - \lambda)^n(w)\neq 0,\quad (\varphi_1-\lambda)^{n+1}(w) = 0.$$
		In this case, $(ad_\Gamma-\lambda)^n(w\ot 1) = (\varphi_1-\lambda)^n(w)\ot 1\neq 0$, and $(ad_\Gamma-\lambda)^{n+ 1}(w\ot 1) = 0$. Therefore $D(\Gamma)\subsetneq F(\Gamma)$$,\Gamma\in \Omega_3(P_1\ot P_2 )\cup\Omega_3'(P_1\ot P_2)$.
	\end{proof}
	Next, we apply the results above to Weyl algebras. It is shown in last section that Weyl algebras satisfy the conditions of the above theorem.
	
	It is known that $p\in \Omega_1(A_1),pq \in \Omega_2(A_1)$. By applying Corollary \ref{coro1} and Corollary \ref{coro3}, we can conclude that $p\ot 1 + 1 \ot p \in \Omega_1(A_2),$ $p\ot 1 + 1\ot pq \in \Omega_3(A_2)$, and also that 
	$$p\ot1\ot 1 + 1\ot p\ot 1 + 1\ot1\ot p \in \Omega_1(A_3).$$
	
	\begin{coro}\label{coro4}
		Let $z_1, z_2\in A_1\setminus K,\Gamma = z_1\ot 1 + 1\ot z_2$. If $z_1$ is nilpotent, $z_2$ is semisimple, then  $C(\Gamma)$ is commutative. If either $z_1$ or $z_2$ belongs to $\Omega_0'$, then $C(\Gamma)$ is also commutative. 
		\end{coro}
	\begin{proof}
		When $z_1$ is nilpotent and $z_2$ is semisimple, or when $z_1 \in \Omega_0'$ and $z_2$ is semisimple, then $F(z_1) = N(z_1), F(z_2) = D(z_2)$; when $z_1 \in \Omega_0'$, and $z_2$ is not semisimple, then $F(z_1) = D(z_1), F(z_2) = N(z_2)$; the same applies when $z_2 \in \Omega_0'$. In summary, $z_1,z_2$ satisfy the conditions of Corollary \ref{coro3}. It is known that the centralizers of any non-constant elements in $A_1$ are commutative {\cite{a}}, and by Corollary \ref{coro3}, $C(z_1 \otimes 1 + 1 \otimes z_2) = C(z_1) \otimes C(z_2)$ is commutative.
	\end{proof}
	When both $z_1$ and $z_2$ are nilpotent  in $A_1$, Corollary \ref{coro4} may not hold. For example, 
	$$ C(p \otimes 1 + 1 \otimes p) = K - \langle q \otimes 1 - 1 \otimes q, p \otimes 1, 1 \otimes p \rangle, $$ the subalgebra of $A_2$ generated by $q \otimes 1 - 1 \otimes q, p \otimes 1, 1 \otimes p$.
	There is also a counterexample when both elements are semisimple: 
	$$ C(pq \otimes 1 + 1 \otimes pq) = K - \langle p \otimes q, pq \otimes 1, 1 \otimes pq, q \otimes p \rangle. $$
	Therefore, the centralizers of noncentral elements in $A_n$ with $n\ge2$ are not necessarily commutative, and in these two counterexamples, $C(z_1) \otimes C(z_2) \subsetneq C(\Gamma)$.
	
	By analogy with Proposition \ref{lemma1}, we can draw the following conclusion.
	
	\begin{coro}
		Let $\varphi,\psi: A_1 \to A_1$ be endomorphisms  such that Im$\,\varphi \subsetneq A_1$ and Im$\,\psi \subsetneq A_1$. Define $\xi = \varphi \otimes \psi: A_2 \rightarrow A_2,v\ot w\mapsto \varphi(v)\ot \psi(w)$. Let $z_1, z_2 \in A_1 \setminus K, \Gamma = z_1 \otimes 1 + 1 \otimes z_2$, then 
		$$ F(\xi(\Gamma)) = F(\varphi(z_1)) \otimes F(\psi(z_2)). $$
		
		If $z_1, z_2$ are nilpotent, then $\xi(\Gamma) \in \Omega_1'(A_2)$;
		if $z_1, z_2$ are semisimple, then $\xi(\Gamma) \in \Omega_2'(A_2)$;
		if one of $z_1, z_2$ is nilpotent and the other is semisimple, then $\xi(\Gamma) \in \Omega_3'(A_2)$.
	\end{coro}
	\begin{proof}
		From Proposition \ref{lemma1}, it is known that if $z_1, z_2$ are nilpotent, then $\varphi(z_1), \psi(z_2) \in \Omega_1'(A_1)$. If $z_1, z_2$ are semisimple, then $\varphi(z_1), \psi(z_2) \in \Omega_2'(A_1)$. Therefore, if both $z_1, z_2$ are nilpotent, we have 
		$$ D(\varphi(z_1)) = C(\varphi(z_1)) \subsetneq N(\varphi(z_1)) = F(\varphi(z_1)) \subsetneq A_1.$$
		The same holds for $\psi(z_2)$. Since $\xi(\Gamma) = \varphi(z_1) \otimes 1 + 1 \otimes \psi(z_2)$, by Corollary \ref{coro1} we have 
		$$ D(\xi(\Gamma)) = C(\xi(\Gamma)) \subsetneq N(\xi(\Gamma)) = F(\xi(\Gamma)) = F(\varphi(z_1)) \otimes F(\psi(z_2)) \subsetneq A_2. $$
		Thus $\xi(\Gamma) \in \Omega_1'(A_2)$. The remaining cases follow similarly.
	\end{proof}

\end{document}